\documentclass[12 pt]{article}
\pdfoutput=1

\usepackage{geometry}
\usepackage[utf8]{inputenc}
\usepackage[T1]{fontenc}

\usepackage{amsmath}
\usepackage{amssymb}
\usepackage{amsthm}
\usepackage{mathtools}
\usepackage{mathrsfs}
\usepackage{dsfont}
\usepackage{stmaryrd}

\usepackage{csquotes} 
\usepackage{xcolor} 
\usepackage{soul} 
\usepackage{framed} 

\usepackage{enumitem} 

\usepackage{graphicx} 
\usepackage{float}
\usepackage{caption}
\usepackage{tikz}
\usetikzlibrary{shapes.geometric}
\usetikzlibrary{arrows}

\usepackage{tikz-cd}

\usepackage{pgfplots}
\pgfplotsset{compat=1.18}

\usepackage[backend=biber,
            style=alphabetic,
            maxnames=10,
            minnames=2,
            maxalphanames=5,
            minalphanames=2,
            backref=true,
            backrefstyle=none]{biblatex}

\usepackage[colorinlistoftodos]{todonotes} 
\usepackage{pdflscape} 

\usepackage{hyperref}
\hypersetup{
    breaklinks=true,
    colorlinks=true,
    linkcolor=cyan,     
    urlcolor={blue!60},
    citecolor=teal,
}

\usepackage[hyphenbreaks]{breakurl} 

\usepackage{adjustbox} 

\usepackage{totcount} 

\usepackage{upgreek} 

\numberwithin{equation}{section}

\theoremstyle{plain} 

\newtheorem{prop}[equation]{Proposition}
\newtheorem{cor}[equation]{Corollary}
\newtheorem{thm}[equation]{Theorem}
\newtheorem{lemma}[equation]{Lemma}
\newtheorem{rmk}[equation]{Remark}

\newtheorem{hypothesis}[equation]{Hypothesis}
\newtheorem{clm}[equation]{Claim}

\theoremstyle{definition}
\newtheorem{defi}[equation]{Definition}

\DeclareMathOperator*{\colim}{colim}

\definecolor{lightblueforque}{HTML}{8fcaff}

\newtheorem{protoerr}{ERROR}

\newtheorem{protoque}[protoerr]{QUESTION}

\newtheorem{prototbe}[protoerr]{To be expanded}

\newtheorem{protoysr}[protoerr]{You should remember}

\definecolor{color1}{rgb}{1,0.84,0}
\definecolor{color2}{rgb}{0.08,0.4,0.75}
\definecolor{color3}{rgb}{0.83,0.19,0.19}
\definecolor{color4}{rgb}{0.23,0.88,0.28}
\definecolor{color5}{rgb}{0.8,0,0.6}

\newcounter{tikzcomd}

\title{Coproduct idempotent algebras over internal operads in enriched $\infty$-categories.}
\author{Federico Ernesto Mocchetti\footnote{Università degli Studi di Milano - Universit{\"a}t Osnabr{\"u}ck}}
\date{\today}

\begin{document}
\maketitle

\begin{abstract}
    {\footnotesize 
    In \cite{Heine2023a}, H. Heine shows that given a symmetric monoidal $\infty$-category $\mathcal{V}$ and a weakly $\mathcal{V}$-enriched monad $T$ over an $\infty$-category $\mathcal{C}$, then there is an induced action of $\mathcal{V}$ on $LMod_T(\mathcal{C})$. Moreover, properties like tensoring or enrichment can be transferred from the action on $\mathcal{C}$ to that on $LMod_T(\mathcal{C})$.
    
    We see that the action of an internal operad $O \in Alg(sSeq(\mathcal{C}))$ can be interpreted as the action of a monad $T_O$, such that $Alg_O(\mathcal{C})\cong LMod_{T_O}(\mathcal{C})$. 
    We can then prove that, under a presentability assumption, if the category $\mathcal{C}$ admits cotensors with respect to the action of $\mathcal{V}$, then so does $Alg_O(\mathcal{C})\cong LMod_{T_O}(\mathcal{C})$.
    This is used to show that the coproduct-idempotent algebras are fixed by the induced tensoring action. 
    We apply this to the stable motivic homotopy category and prove that the tensor of any motivic sphere with rational motivic cohomology is equivalent to the latter.

    }
    
\end{abstract}

\thispagestyle{empty}

\setcounter{tikzcomd}{0}
\setcounter{section}{-1}

\section{Introduction}

Let S be a nice scheme (for instance, quasi-separated and quasi-compact) and let $R \in \mathcal{SH}(S)$ be a ring spectrum in the motivic stable homotopy category over it. One can define a motivic version of Hochschild homology \cite{DHOO2022} via the (derived) smash product:
\[
MHH(R) = R \wedge_{R\wedge R^{op}} R.
\]

In the event $R \in Alg_{\mathbb{E}_{\infty}}(\mathcal{SH}(S))$ is a highly commutative ring spectrum, one has equivalences:
\begin{gather*}
    R \cong R^{op} \\
    R \cong |Bar(R)|=|\cdots \to R \wedge R \to R \to \mathbb{S}|.
\end{gather*}
Here, the vertical bars $| - |$ indicate the geometric realisation functor from simplicial objects in an infinity category to the infinity category itself, in other words, the colimit along the indexing category of the objects sitting in the various degrees. 
$Bar(-)$ denotes the standard resolution.

Then we have that:
\[
MHH(R) \cong R \wedge_{R \wedge R} R \cong |R \wedge_{R \wedge R} Bar(R)| \cong S^1 \otimes R
\]
is the geometric realisation of the cyclic bar complex: in degree $n$ it presents the $(n+1)$-fold coproduct of $R$ with itself in the category of $\mathbb{E}_{\infty}$ ring spectra; boundary maps are given according to the simplicial model of $S^1$ as $\Delta^1/\partial \Delta^1$, the 1-simplex modulo its boundary \cite[Proposition IV.2.2]{NikSch2018}.

This expression immediately generalises: one defines the tensor of a (small) simplicial set and an $\mathbb{E}_{\infty}$ motivic ring spectrum as the geometric realisation of the associated simplicial object, in other words, as a colimit along the category $\Delta^{op}$. The goal of this paper is to further extend this to an action of all motivic spaces (in which simplicial sets embed as constant presheaves) on highly commutative motivic ring spectra. This will allow us to prove that any idempotent highly commutative motivic ring spectra are fixed by the extended action. To do so, we move into a much broader setting and prove an analogous result for general left-tensored $\infty$-categories.

In section \ref{sec:fun.of.mods}, we compare the left module categories arising from different actions on a certain $\infty$-category. In particular, we prove that given $f:\mathcal{C}^{\otimes}\to \mathcal{C'}^{\otimes}$ monoidal, with $\mathcal{C}^{\otimes}$ and $\mathcal{C'}^{\otimes}$ acting compatibly on a third infinity category $\mathcal{M}$, then for each $A \in Alg_{Assoc}(\mathcal{C})$, $LMod_A(\mathcal{M}) \cong LMod_{f(A)}(\mathcal{M})$ (theorem \ref{thm:LModA=LModf(A)}).

In section \ref{sec:enr.adj} we apply the previous result. We introduce the notion of operads internal to a symmetric monoidal $\infty$-category and associate to each a monad. The results from the previous section prove an identification between the $\infty$-category of algebras over an operad and left modules over the associated monad. We then move to the context of enriched categories, operads and monads. 
We investigate when different kinds of enrichment or actions can be transported from the base category to the category of modules (proposition \ref{prop:tensor.algebras} and corollary \ref{cor:cotensors.on.algo(C)}).

Section \ref{sec:idemp.enr} focuses on the proof of the following result: an algebra over an internal operad which is coproduct idempotent is invariant with respect to the induced action (theorem \ref{thm:X.otimes.R.cong.R}).

Finally, in the last section, we apply this result to the stable motivic homotopy category: the previous results allow us to recover an action of the unstable motivic homotopy category on highly commutative motivic ring spectra. As an example application of theorem \ref{thm:X.otimes.R.cong.R}, we show that the action of any motivic sphere on rational motivic cohomology is trivial (corollary \ref{cor:MQ.otimes.anything}).

\vspace{1cm}

My deepest thanks go to my supervisors Paul Arne {\O}stv{\ae}r and Markus Spitzweck for proposing this topic and offering ongoing support and guidance throughout the preparation of this paper. I also wish to express my heartfelt thanks to Hadrian Heine for numerous discussions and clarifications regarding the technical aspects of enriched $\infty$-categories.
\clearpage

\tableofcontents

\vspace{1.5cm}

\section{A functor of modules.}\label{sec:fun.of.mods}

Let $f: \mathcal{C}^{\otimes} \to \mathcal{C'}^{\otimes}$ be a monoidal functor, compatible with the actions of $\mathcal{C}^{\otimes}$ on some category $\mathcal{M}$ and of $\mathcal{C'}^{\otimes}$ on some category $ \mathcal{M'}$; suppose, moreover, that there is an isomorphism of simplicial sets $\mathcal{M}\xrightarrow{\sim} \mathcal{M'}$. Let $A \in Alg_{Assoc}(\mathcal{C})$ be an associative algebra in $\mathcal{C}$. The goal of this section is to show that $LMod_A(\mathcal{M}) \simeq LMod_{f(A)}(\mathcal{M'})$.

To address the problem, it is first necessary to specify what \textit{compatible actions} means in higher category theory. In particular, recall that in this context the categories $LMod_A(\mathcal{M})$ and $LMod_{f(A)}(\mathcal{M'})$ are just fibres above a certain associative algebra or more general module categories $LMod(\mathcal{M})$ and $LMod(\mathcal{M'})$; the notion should then be formulated at such level at least. We will be using both the standard and the planar operad formalisms, denoting the first with a $\otimes$ and the second with a $\varoast$; the equivalence between these two formalisms for modules in the $\infty$-categorical setting can be found in \cite[Sections 4.2.1 and 4.2.2]{Lurie2017}.

Let $F:\mathcal{O}^{\otimes} \to \mathcal{O}'^{\otimes}$ be a fibration of $\mathcal{LM}$ operads. Let $\mathcal{C}^{\otimes}=\mathcal{O}^{\otimes} \times_{\mathcal{LM}^{\otimes}} Assoc^{\otimes}$, $\mathcal{C'}^{\otimes}=\mathcal{O'}^{\otimes} \times_{\mathcal{LM}^{\otimes}} Assoc^{\otimes}$ be the underlying monoidal categories and $\mathcal{M}=\mathcal{O}^{\otimes} \times_{\mathcal{LM}^{\otimes}} \{m\}$ and $\mathcal{M}'=\mathcal{O}'^{\otimes} \times_{\mathcal{LM}^{\otimes}} \{m\}$ be the underlying weakly enriched categories. The universal property of the pullback induces a commutative diagram:

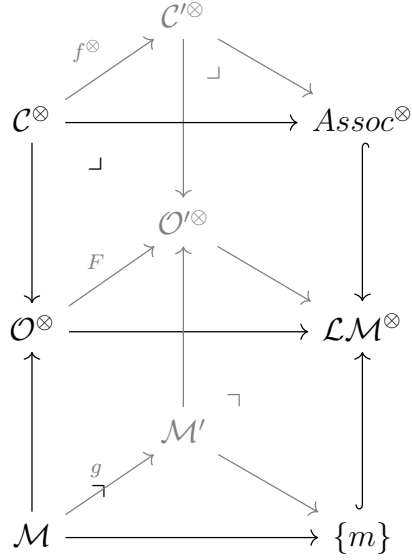
\[\begin{tikzcd}
    &\color{gray}\mathcal{C}'^{\otimes} 
        \arrow[rd, gray] 
        \arrow[dd, gray] &\\
    \mathcal{C}^{\otimes} 
        \arrow[ur, gray, "f^{\otimes}"] 
        \arrow [rr] 
        \arrow[dd] & &
    Assoc^{\otimes} 
        \arrow[dd, hook]\\
    & \color{gray}\mathcal{O}'^{\otimes} 
        \arrow[rd, gray] & \\
    \mathcal{O}^{\otimes} 
        \arrow[rr] 
        \arrow[ur, gray, "F"] & &
    \mathcal{LM}^{\otimes}
        \arrow[uull, phantom, "\lrcorner", very near end] 
        \arrow[uuul, phantom, "\lrcorner", very near end, gray] 
        \arrow[ddll, phantom, "\urcorner", very near end]
        \arrow[dl, phantom, "\urcorner", very near end, gray]\\
    &\color{gray}\mathcal{M}'
        \arrow[rd, gray] 
        \arrow[uu, gray] &\\
    \mathcal{M} 
        \arrow[ur, gray, "g"]
        \arrow[rr] 
        \arrow[uu] &&
    \{m\} 
        \arrow[uu, hook]
\end{tikzcd}\addtocounter{tikzcomd}{1}\]

Passing to the planar operad description \cite[Section 4.1.3. and Notation 4.2.2.17.]{Lurie2017}, one obtains a commutative diagram:

\begin{equation} \label{eq:square}
\begin{tikzcd} 
    \mathcal{M}^{\varoast} \arrow[r, "g^{\varoast}"] \arrow[d] & \mathcal{M}'^{\varoast} \arrow [d]\\
    \mathcal{C}^{\varoast} \arrow[d] \arrow[r, "f^{\varoast}"] & \mathcal{C}'^{\varoast} \arrow[dl]\\
    N(\Delta)^{op} &
\end{tikzcd}\addtocounter{tikzcomd}{1}
\end{equation}

Recall that $\mathcal{C}^{\varoast}$ is defined in terms of $\mathcal{C}^{\otimes}$ (or of $\mathcal{O}^{\otimes}$) via the pullback diagram:
\[\begin{tikzcd}
    \mathcal{C}^{\varoast}
        \arrow[r]
        \arrow[d]&
    N(\Delta)^{op}
        \arrow[d,"Cut"]\\
    \mathcal{C}^{\otimes}
        \arrow[r]
        \arrow[d]&
    Assoc^{\otimes} 
        \arrow[d, hook]\\
    \mathcal{O}^{\otimes} 
        \arrow[r]&
    \mathcal{LM}^{\otimes}
        \arrow[ul, phantom, "\lrcorner", very near end]
        \arrow[uul, phantom, "\lrcorner", very near end]\\
\end{tikzcd}\addtocounter{tikzcomd}{1}\]
where the functor $Cut$ is described in \cite[Construction 4.1.2.9]{Lurie2017}. The arrow $f^{\varoast}: \mathcal{C}^{\varoast} \to \mathcal{C}'^{\varoast}$ over $N(\Delta)^{op}$ is a consequence of the universal property of the pullback:
\[\begin{tikzcd} 
    &\color{gray}\mathcal{C'}^{\varoast}
        \arrow[dd, gray]
        \arrow[rd, gray]&\\
    \mathcal{C}^{\varoast}
        \arrow[rr]
        \arrow[dd]
        \arrow[ru, gray, "f^{\varoast}"]&&
    N(\Delta)^{op}
        \arrow[dd,"Cut"]\\
    &\color{gray}\mathcal{C}'^{\otimes} 
        \arrow[rd, gray]  &\\
    \mathcal{C}^{\otimes} 
        \arrow[ur, gray, "f^{\otimes}"] 
        \arrow [rr]& &
    Assoc^{\otimes}\\
\end{tikzcd}\addtocounter{tikzcomd}{1}\]

On the other hand, $\mathcal{M}^{\varoast}$ is defined in terms of a universal property \cite[Notation 4.2.2.17]{Lurie2017}; one first defines $\overline{\mathcal{M}}^{\varoast}$ as the simplicial set over $N(\Delta)^{op}$ satisfying:
\begin{equation}\label{eq:univ.prop.M}
    Hom_{sSet_{/N(\Delta)^{op}}}(K, \overline{\mathcal{M}}^{\varoast}) \cong Hom_{sSet_{/\mathcal{LM}^{\otimes}}}(K \times \Delta^1, \mathcal{O}^{\otimes})
\end{equation}
where the map $K \times \Delta^1 \to \mathcal{LM}^{\otimes}$ is given by the composition:
\[
    K \times \Delta^1 \to N(\Delta^{op}) \times \Delta^1 \xrightarrow{\gamma} \mathcal{LM}^{\otimes}
\]
and $\gamma$ represents a natural transformation from the functor $LCut:N(\Delta^{op}) \to \mathcal{LM}^{\otimes}$ to the functor $Cut:N(\Delta^{op}) \to \mathcal{LM}^{\otimes}$, see \cite[Remark 4.2.2.8]{Lurie2017} for more details. Observe that vertices in $\overline{\mathcal{M}}^{\varoast}$ (i.e. morphisms $\Delta^0 \to \overline{\mathcal{M}}^{\varoast}$) correspond to one-simplices, or arrows, in $\mathcal{O}^{\otimes}$. $\mathcal{M}^{\varoast}$ is given by the sub-simplicial set of $\overline{\mathcal{M}}^{\varoast}$ spanned by vertices corresponding to inert morphisms in $\mathcal{O}^{\otimes}$. The morphism of operads:
\[\begin{tikzcd} 
    \mathcal{O}^{\otimes} \arrow[rr, "F"] \arrow[rd] &&
    \mathcal{O}'^{\otimes} \arrow[ld] \\
    & \mathcal{LM}^{\otimes}&
\end{tikzcd}\addtocounter{tikzcomd}{1}\]
induces then the arrow over $N(\Delta)^{op}$:
\[\begin{tikzcd} 
    \mathcal{M}^{\varoast} \arrow[rr, "g^{\varoast}"] \arrow[rd] &&
    \mathcal{M}'^{\varoast} \arrow[ld] \\
    & N(\Delta)^{op}&
\end{tikzcd}\addtocounter{tikzcomd}{1}\]
By construction, a simplex in $\mathcal{M}^{\varoast}$ corresponds to a pair of simplices in $\mathcal{O}^{\otimes}$ (given by the endpoints of $\Delta^1$), one of which of ``module'' type (in other words, associated to the functor $LCut$), and the other of algebra type (associated to $Cut$). At the level of points, this produces a module object and its underlying algebra. Similar facts hold for $\mathcal{M}'^{\varoast}$. The vertical arrows $\mathcal{M}^{\varoast} \to \mathcal{C}^{\varoast}$ and $\mathcal{M}'^{\varoast} \to \mathcal{C}'^{\varoast}$, given as in \cite[Remark 4.2.2.19]{Lurie2017}, send a simplex in the module category to the algebra part of that simplex.
It follows from the definitions that the square in \ref{eq:square} is commutative.

Module categories can be determined as follows \cite[Remark 4.2.2.19]{Lurie2017}: given $\mathcal{M}^{\varoast} \xrightarrow{q} \mathcal{C}^{\varoast} \xrightarrow{p} N(\Delta)^{op}$, then:
\begin{itemize}
    \item $Alg_{\mathbb{A}_{\infty}}(\mathcal{C})$ is the full subcategory of $Fun_{N(\Delta)^{op}}({N(\Delta)^{op}}, \mathcal{C}^{\varoast})$ of sections of $p$ that preserve inert morphisms \cite[Definition 4.1.3.16]{Lurie2017}
    \item $LMod^{\mathbb{A}_{\infty}}(\mathcal{M})$ is given by the full subcategory of  $Fun_{N(\Delta)^{op}}({N(\Delta)^{op}}, \mathcal{M}^{\varoast})$ of sections of $p \circ q$ corresponding via \ref{eq:univ.prop.M} to a certain full subcategory of $ Fun_{\mathcal{LM}^{\otimes}}({N(\Delta)^{op}} \times \Delta^1, \mathcal{O}^{\otimes})$ described in \cite[Definition 4.2.2.10]{Lurie2017}, see also definition \ref{def:LMod(M)} below.
    \item The inclusion $N(\Delta)^{op} \times \{1\} \hookrightarrow N(\Delta)^{op} \times \Delta^1$ induces a forgetful functor $LMod^{\mathbb{A}_{\infty}}(\mathcal{M})\to Alg_{\mathbb{A}_{\infty}}(\mathcal{C})$. Given $A \in Alg_{\mathbb{A}_{\infty}}(\mathcal{C})$, we denote:
    \[
        LMod^{\mathbb{A}_{\infty}}_A(\mathcal{M})= LMod^{\mathbb{A}_{\infty}}(\mathcal{M})\times_{Alg_{\mathbb{A}_{\infty}}(\mathcal{C})} \{A\}.
    \]
    This can be identified via \ref{eq:univ.prop.M} with a certain full subcategory of the functor category:
    \[
    Fun_{N(\Delta)^{op}}(N(\Delta)^{op}, N(\Delta)^{op}\times_{\mathcal{C}^{\varoast}} \mathcal{M}^{\varoast}).
    \]
    
\end{itemize}
This description suits in a diagram of the form:
\[\begin{tikzcd} 
        &N(\Delta)^{op} 
            \arrow[ld, swap, "(2)"] 
            \arrow[d, "(1)"]\\
        \mathcal{M}^{\varoast} \times_{\mathcal{C}^{\varoast}} N(\Delta)^{op} 
            \arrow[r] 
            \arrow[d] & 
        \mathcal{M}^{\varoast}
            \arrow[d, "q"]\\
        N(\Delta)^{op} 
            \arrow[rd, "id"] 
            \arrow[r, "A"]  &
        \mathcal{C}^{\varoast} 
            \arrow[d, "p"]
            \arrow[ul, phantom, "\lrcorner", very near end] \\
        & N(\Delta)^{op} 
\end{tikzcd}\addtocounter{tikzcomd}{1}\]
The objects of $LMod^{\mathbb{A}_{\infty}}(\mathcal{M})$ are then (some) arrows fitting in $(1)$, while (some) arrows fitting in place $(2)$ give the objects of $LMod_A^{\mathbb{A}_{\infty}}(\mathcal{M})$.

\begin{thm}\label{thm:LModA=LModf(A)} 
    In the above situation, suppose further that:
    \begin{itemize}
        \item the maps $\mathcal{O}^{\otimes} \to \mathcal{LM}^{\otimes}$ and $\mathcal{O}'^{\otimes} \to \mathcal{LM}^{\otimes}$ are co-Cartesian fibrations of $\infty$-operads, so that $\mathcal{M}$ is left-tensored on $\mathcal{C}^{\otimes}$ and $\mathcal{M}'$ is left-tensored on $\mathcal{C}'^{\otimes}$,
        \item the functor $g:\mathcal{M} \to \mathcal{M}'$ gives rise to an isomorphism of simplicial sets.
    \end{itemize}  
    Then for each associative algebra $A \in Alg(\mathcal{C})$, we have an equivalence $LMod_A(\mathcal{M}) \cong LMod_{f(A)}(\mathcal{M}')$.
\end{thm}

\begin{proof}
Consider the diagram
\begin{equation}\label{diag:LModA(M)}
\begin{tikzcd} 
    &N(\Delta)^{op} 
        \arrow[ld, swap, "Y"] 
        \arrow[d, "X"]
        \arrow[rd]&\\
    \mathcal{M}^{\varoast} \times_{\mathcal{C}^{\varoast}} N(\Delta)^{op} 
        \arrow[r, "i"] 
        \arrow[d, "j"] & 
    \mathcal{M}^{\varoast}
        \arrow[d, "q"]
        \arrow[r, "g^{\varoast}"]&
    \mathcal{M}'^{\varoast}
        \arrow[d, "q'"]\\
    N(\Delta)^{op} 
        \arrow[rd] 
        \arrow[r, "A"]  &
    \mathcal{C}^{\varoast} 
        \arrow[d, "p"]
        \arrow[ul, phantom, "\lrcorner", very near end]
        \arrow[r, "f^{\varoast}"]&
    \mathcal{C}'^{\varoast}
        \arrow[ld, "p'"]\\
    & N(\Delta)^{op}& 
\end{tikzcd}\addtocounter{tikzcomd}{1}    
\end{equation}
Our aim is to identify $\mathcal{M}^{\varoast} \times_{\mathcal{C}^{\varoast}} N(\Delta)^{op}$ with the pullback of $\mathcal{M}'^{\varoast}$ and  $N(\Delta)^{op}$  over $\mathcal{C}'^{\varoast}$ along the maps $q'$ and $f^{\varoast}(A)$.
To do so, we show that the square:
\[\begin{tikzcd}
    \mathcal{M}^{\varoast}
        \arrow[d, "q"]
        \arrow[r, "g^{\varoast}"]&
    \mathcal{M}'^{\varoast}
        \arrow[d, "q'"]\\
    \mathcal{C}^{\varoast}
        \arrow[r, "f^{\varoast}"]&
    \mathcal{C}'^{\varoast}\\
\end{tikzcd}\addtocounter{tikzcomd}{1}\]
is a pullback itself. We use the following ``co''-version of \cite[Corollary 2.4.4.4]{Lurie2009}:
\begin{cor}\label{cor:lurie.2.4.4.4}
Suppose we are given maps $\mathcal{C} \xrightarrow{r} \mathcal{D} \xrightarrow{s} \mathcal{E}$ of $\infty$-categories such that both $s$ and $s \circ r$ are locally co-Cartesian fibrations. Suppose that $r$ carries locally ($s \circ r$)-co-Cartesian edges of $\mathcal{C}$ to locally $s$-co-Cartesian edges of $\mathcal{D}$ and that for every object $Z \in \mathcal{E}$, the induced map $r_Z:\mathcal{C}_Z \to \mathcal{D}_Z$ is a categorical equivalence. Then $r$ is a categorical equivalence.
\end{cor}

This is applied to the maps $\mathcal{M}^{\varoast} \xrightarrow{r} \mathcal{M}'^{\varoast} \times_{\mathcal{C}'^{\varoast}}\mathcal{C}^{\varoast} \xrightarrow{s} \mathcal{C}^{\varoast}$, where $r$ and $s$ are the canonical maps to and from the pullback. Notice that $s \circ r \simeq q$. There is in fact a diagram:
\[\begin{tikzcd}
    \mathcal{M}^{\varoast}
        \arrow[ddr, swap, bend right=30, "q"]
        \arrow[drr, bend left=20, "g^{\varoast}"]
        \arrow[dr, "r"]&&\\
    &\mathcal{M}'^{\varoast} \times_{\mathcal{C}'^{\varoast}}\mathcal{C}^{\varoast}
        \arrow[r]
        \arrow[d, "s"]
    &\mathcal{M}'^{\varoast}
        \arrow[d, "q'"]\\
    &\mathcal{C}^{\varoast}
        \arrow[r, "f^{\varoast}"]&
    \mathcal{C}'^{\varoast}\\
\end{tikzcd}\addtocounter{tikzcomd}{1}\]

The maps $q\simeq s \circ r$ and $q'$ are locally co-Cartesian fibrations after \cite[Lemma 4.2.2.20]{Lurie2017}. Locally co-Cartesian fibrations are closed under pullback \cite[\href{https://kerodon.net/tag/01V1}{Remark 5.1.5.5}]{Lurie2018}, so $s$ is a locally co-Cartesian fibration as well.

Let now $e:\Delta^1 \to \mathcal{M}^{\varoast}$ be a locally $q$-co-Cartesian edge. By \cite[Definition 5.1.3.1]{Lurie2018}, this means that for every commutative diagram:
\[\begin{tikzcd}
    &\Delta^1 
        \arrow[ld, bend right=10]
        \arrow[d]
        \arrow[rd, "e"]&&\\
    \Lambda^n_0
        \arrow[d]
        \arrow[r]&
    \mathcal{M}^{\varoast}\times_{\mathcal{C}^{\varoast}} \Delta^1 
        \arrow[r]
        \arrow[d]&
    \mathcal{M}^{\varoast}
        \arrow[d, "q"]\\
    \Delta^n
        \arrow[r]
        \arrow[ur, dashed]&
    \Delta^1
        \arrow[r, "q(e)"]&
    \mathcal{C}^{\varoast}
        \arrow[ul, phantom, "\lrcorner", very near end]\\
\end{tikzcd}\addtocounter{tikzcomd}{1}\]
the dotted lift exists, where the pullback $\mathcal{M}^{\varoast}\times_{\mathcal{C}^{\varoast}} \Delta^1$ is formed in simplicial sets and the edge $\Delta^1 \to \mathcal{M}^{\varoast}\times_{\mathcal{C}^{\varoast}} \Delta^1$ is the unique lift of $e$ with non-trivial image in $\Delta^1$. The inclusion $\Delta^1 \to \Lambda^n_0$ is given by: $\Delta^1=\Delta^{\{0,1\}} \to \Lambda^n_0$.

It is worth having a more explicit description of the pullback $\mathcal{M}^{\varoast}\times_{\mathcal{C}^{\varoast}} \Delta^1$. The objects in $\mathcal{M}^{\varoast}$ and in $\mathcal{C}^{\varoast}$ are presented according to the decompositions of \cite[Remark 4.2.2.18]{Lurie2017} and \cite[Notation 4.1.3.5]{Lurie2017}. If the arrow $e$ lies over a map $\alpha: [k] \to [n]$ in $\Delta$, it can be described as:
\begin{equation} \label{eq:e.arrow.in.M}
\begin{gathered}
    e:(C_1, \ldots, C_n, M) \to (D_1, \ldots, D_k, N) \text{ for } C_i,\, D_j \in \mathcal{C},\, M,\,N \in \mathcal{M}\\
    C_{\alpha(j-1)+1} \otimes \cdots \otimes C_{\alpha(j)} \to D_j \text{ in } \mathcal{C} \text{ for } 1 \leq j \leq k \\
    C_{\alpha(k)+1} \otimes \cdots \otimes C_{n} \otimes M \to N \text{ in } \mathcal{M}
\end{gathered}
\end{equation}
$q(e)$ is then given by the arrow:
\[
    q(e):(C_1, \ldots, C_n) \to (D_1, \ldots, D_k)
\]
in $\mathcal{C}^{\varoast}$ described by the same arrows in $\mathcal{C}$ as above.

Pullbacks in simplicial sets can be evaluated level-wise \cite[\href{https://kerodon.net/tag/000J}{Remark 1.1.1.7}]{Lurie2018}. At the level of objects, one has:
\[
    (\mathcal{M}^{\varoast}\times_{\mathcal{C}^{\varoast}} \Delta^1)_0 \cong \mathcal{M}^{\varoast}_0\times_{\mathcal{C}^{\varoast}_0} \Delta^1_0
\]
The description of the map $q(e)$ above implies that this consists of the disjoint union of two copies of $\mathcal{M}$, namely the fibre over $(C_1, \ldots, C_n)$, corresponding to $0 \in \Delta^1_0$, and that over $(D_1, \ldots, D_k)$, corresponding to $1 \in \Delta^1_0$.

At the level of 1-simplices, $\Delta^1_1$ presents a single non-trivial 1-simplex, the one that is sent to $q(e)$. Hence in this level of the pullback:
\[
    (\mathcal{M}^{\varoast}\times_{\mathcal{C}^{\varoast}} \Delta^1)_1 \cong \mathcal{M}^{\varoast}_1\times_{\mathcal{C}^{\varoast}_1} \Delta^1_1
\]
there are (among others) arrows from the first copy of $\mathcal{M}$ to the second copy of $\mathcal{M}$ that lie over $q(e)$. These can be described in $\mathcal{M}^{\varoast}$ as arrows:
\begin{gather*}
    \tilde{e}:(C_1, \ldots, C_n, A) \to (D_1, \ldots, D_k, B) \text{ for } A,\,B \in \mathcal{M}\\
     C_{\alpha(j-1)+1} \otimes \cdots \otimes C_{\alpha(j)} \to D_j \text{ in } \mathcal{C} \text{ for } 1 \leq j \leq k \\
    C_{\alpha(k)+1} \otimes \cdots \otimes C_{n} \otimes A \to B \text{ in } \mathcal{M}\\
    q(\tilde{e})=e \text{ in } \mathcal{C}^{\varoast}
\end{gather*}
where the arrows internal to $\mathcal{C}$ are the same as in \ref{eq:e.arrow.in.M}. 
By identifying the fibres over $0$ and $1$ with $\mathcal{M}$, the arrow $\tilde{e}$ corresponds to the arrows (in $\mathcal{M}$):
\begin{equation}\label{eq:1.simp.of.M.x.D1}
    A \xrightarrow{q(e)_!} C_{\alpha(k)+1} \otimes \cdots \otimes C_{n} \otimes A \to B
\end{equation}

As trivial 1-simplices in $\Delta^1$ correspond to arrows internal to one or the other copy of $\mathcal{M}$, one may conclude that all one simplices of $\mathcal{M}^{\varoast}\times_{\mathcal{C}^{\varoast}} \Delta^1$ are given by certain (specific) diagrams in $\mathcal{M}$. 
A similar description can be obtained for the higher simplices as well, see figure \ref{fig:simplices.in.Delta1.times.M} for some examples.

\begin{figure}
\centering
\resizebox{\textwidth}{!}{ 
\begin{tikzpicture}[line cap=round, line join=round, >=triangle 45]
    
\begin{scope}
    \fill[line width=2pt, color=color2, fill=color2, fill opacity=0.1] (-1, 6) -- (1, 6) -- (1, -1) -- (-1, -1) -- cycle;
    \fill[line width=2pt, color=color4, fill=color4, fill opacity=0.2] (3.330127018922192, 6) -- (5.330127018922195, 6) -- (5.330127018922194, -1) -- (3.3301270189221936, -1) -- cycle;
    \filldraw[line width=2pt, color=color3, fill=color3, fill opacity=1] (10, 2.8) -- (10, 2.2) -- (10.519615242270662, 2.5) -- cycle;
    \fill[line width=0pt, color=color1, fill=color1, fill opacity=0.6] (12, 5) -- (12, 0) -- (17, 0) -- (22, 5) -- cycle;
    \fill[line width=0pt,  color=color1,  fill=color1, fill opacity=0.6] (0, 5) -- (0, 0) -- (4.330127018922194, 2.5) -- cycle;
            
    \draw[line width=2pt, color=color3, smooth, samples=100, domain=6:10] plot(\x, {sin((2*3.14*(\x))*180/pi)/5+2.5});

    \begin{small}
    \draw [->, line width=1pt] (0, 5) -- (0, 0) node[left=3pt,  midway] {$\alpha$};
    \draw [->, line width=1pt] (0, 0) -- (4.330127018922194, 2.5) node[below=3pt,  midway] {$\beta$};
    \draw [->, line width=1pt] (0, 5) -- (4.330127018922194, 2.5) node[above=3pt,  midway] {$\gamma$};
    \draw [->, line width=1pt] (12, 5) -- (12, 0) node[left=3pt,  midway] {$\alpha$};
    \draw [->, line width=1pt] (12, 0) -- (17, 0) node[above=3pt,  midway] {$q(e)_!$};
    \draw [->, line width=1pt] (17, 0) -- (22, 5) node[below right=3pt,  midway] {$\beta$};
    \draw [->, line width=1pt] (12, 5) -- (17, 5) node[below=3pt,  midway] {$q(e)_!$};
    \draw [->, line width=1pt] (17, 5) -- (17, 0)node[left=3pt,  midway] {$q(e)_!\alpha$};
    \draw [->, line width=1pt] (17, 5) -- (22, 5) node[below=3pt,  midway] {$\gamma$};
    
    \draw [fill=color2] (0, 0) circle (4pt) node[left=3pt, color2] {$(C_1,  \ldots,  C_n,  M_1)$};
    \draw [fill=color2] (0, 5) circle (4pt) node[left=3pt, color2] {$(C_1,  \ldots,  C_n,  M_0)$};
    \draw [fill=color2] (4.330127018922194, 2.5) circle (4pt) node[above right=2pt, color2] {$(D_1,  \ldots,  D_k,  M_2)$};
    \draw [fill=color2] (12, 5) circle (4pt) node[left=3pt, color2] {$M_0$};
    \draw [fill=color2] (12, 0) circle (4pt) node[left=3pt, color2] {$M_1$};
    \draw [fill=color2] (17, 0) circle (4pt) node[below=3pt, color2] {$C_{\alpha(k)+1}\otimes \cdots \otimes C_n \otimes M_1$};
    \draw [fill=color2] (17, 5) circle (4pt) node[above=3pt, color2] {$C_{\alpha(k)+1}\otimes \cdots \otimes C_n \otimes M_0$};
    \draw [fill=color2] (22, 5) circle (4pt) node[right=3pt, color2] {$M_2$};
    \end{small}
    
\end{scope}

\begin{scope}[yshift=-8cm]

    \filldraw[line width=2pt,color=color3,fill=color3,fill opacity=1] (10,2.8) -- (10,2.2) -- (10.519615242270662,2.5) -- cycle;
    \fill[line width=2pt,color=color2,fill=color2,fill opacity=0.1] (-1,6) -- (1,6) -- (1,-1) -- (-1,-1) -- cycle;
    \fill[line width=2pt,color=color4,fill=color4,fill opacity=0.2] (3.330127018922193,-1) -- (5.330127018922193,-1) -- (5.330127018922193,6) -- (3.330127018922193,6) -- cycle;
    \fill[line width=2pt,color=color1,fill=color1,fill opacity=0.6] (17,2.5) -- (21.33012701892219,0) -- (21.330127018922198,5) -- cycle;
    \fill[line width=2pt,color=color1,fill=color1,fill opacity=0.6] (4.330127018922193,0) -- (4.330127018922193,5) -- (0,2.5) -- cycle;
    
    \draw[line width=2pt,color=color3,smooth,samples=100,domain=6:10] plot(\x,{sin((2*3.14*(\x))*180/pi)/5+2.5});
    
    \begin{small}        
    \draw [->,line width=1pt] (0,2.5) -- (4.330127018922193,0) node[below=3pt,  midway] {$\alpha$};
    \draw [->,line width=1pt] (0,2.5) -- (4.330127018922193,5) node[above=3pt,  midway] {$\gamma$};
    \draw [->,line width=1pt] (4.330127018922193,0) -- (4.330127018922193,5) node[left=3pt,  midway] {$\beta$};
    \draw [->,line width=1pt] (12,2.5) -- (17,2.5) node[below=3pt,  midway] {$q(e)_!$};
    \draw [->,line width=1pt] (17,2.5) -- (21.33012701892219,0) node[below=3pt,  midway] {$\alpha$};
    \draw [->,line width=1pt] (21.33012701892219,0) -- (21.330127018922198,5) node[right=3pt,  midway] {$\beta$};
    \draw [->,line width=1pt] (17,2.5) -- (21.330127018922198,5) node[above=3pt,  midway] {$\gamma$};
    
    \draw [fill=color2] (0,2.5) circle (4pt) node[left=3pt, color2] {$(C_1,  \ldots,  C_n,  M_0)$};
    \draw [fill=color2] (4.330127018922193,0) circle (4pt) node[right=2pt, color2] {$(D_1,  \ldots,  D_k,  M_1)$};
    \draw [fill=color2] (4.330127018922193,5) circle (4pt) node[right=2pt, color2] {$(D_1,  \ldots,  D_k,  M_2)$};
    \draw [fill=color2] (12,2.5) circle (4pt) node[above=2pt, color2] {$M_0$};
    \draw [fill=color2] (17,2.5) circle (4pt) node[above=3pt, color2] {$C_{\alpha(k)+1}\otimes \cdots \otimes C_n \otimes M_0$};
    \draw [fill=color2] (21.33012701892219,0) circle (4pt) node[above=3pt, color2] {$M_1$};
    \draw [fill=color2] (21.330127018922198,5) circle (4pt) node[above=3pt, color2] {$M_2$};
    \end{small}
    
    \draw (0,-2) circle (0pt);
    
\end{scope}

\end{tikzpicture}
}

\resizebox{\textwidth}{!}{
\begin{tikzpicture}[line cap=round, line join=round, >=triangle 45]  

    \filldraw[line width=2pt,color=color3,fill=color3,fill opacity=1] (10,2.8) -- (10,2.2) -- (10.519615242270662,2.5) -- cycle;
    \fill[line width=2pt,color=color2,fill=color2,fill opacity=0.1] (-1,6) -- (1,6) -- (1,-1) -- (-1,-1) -- cycle;
    \fill[line width=2pt,color=color4,fill=color4,fill opacity=0.2] (4,-1) -- (6,-1) -- (6,6) -- (4,6) -- cycle;
    \fill[line width=2pt,color=color5,fill=color5,fill opacity=0.1] (0,5) -- (0,0) -- (5,0) -- (5,5) -- cycle;
    \fill[line width=2pt,color=color5,fill=color5,fill opacity=0.1] (17,5) -- (17,0) -- (22,0) -- (22,5) -- cycle;
    \fill[line width=2pt,color=color1,fill=color1,fill opacity=0.6] (12,0) -- (17,0) -- (17,5) -- (12,5) -- cycle;
    
    \draw[line width=2pt,color=color3,smooth,samples=100,domain=7:10] plot(\x,{sin((2*3.141592653589793*(\x))*180/pi)/5+2.5});
    
    \begin{small}
    \draw [->,line width=1pt] (0,5) -- (0,0) node[left=3pt,  midway] {$\alpha$};
    \draw [->,line width=1pt] (0,5) -- (5,0) node[left=3pt,  near start] {$\gamma$};
    \draw [->,line width=1pt] (0,5) -- (5,5) node[above=3pt,  midway] {$\delta$};
    \draw [->,line width=1pt] (0,0) -- (5,0) node[below=3pt,  midway] {$\beta$};
    \draw [->,line width=1pt] (0,0) -- (5,5) node[left=3pt,  near start] {$\eta$};
    \draw [->,line width=1pt] (5,0) -- (5,5) node[right=3pt,  midway] {$\theta$};
    
    \draw [->,line width=1pt] (12,5) -- (12,0) node[left=3pt,  midway] {$\alpha$};
    \draw [->,line width=1pt] (12,5) -- (17,5) node[below=3pt,  midway] {$q(e)_!$};
    \draw [->,line width=1pt] (17,5) -- (22,5) node[below=3pt,  midway] {$\delta$};
    \draw [->,line width=1pt] (17,5) -- (17,0) node[left=3pt,  midway] {$q(e)_!\alpha$};
    \draw [->,line width=1pt] (12,0) -- (17,0) node[above=3pt,  midway] {$q(e)_!$};
    \draw [->,line width=1pt] (17,5) -- (22,0) node[left=3pt,  near start] {$\gamma$};
    \draw [->,line width=1pt] (17,0) -- (22,0) node[above=3pt,  midway] {$\beta$};
    \draw [->,line width=1pt] (22,0) -- (22,5) node[right=3pt,  midway] {$\theta$};
    \draw [->,line width=1pt] (17,0) -- (22,5) node[left=3pt,  near start] {$\eta$};
    
    \draw [fill=color2] (0,0) circle (4pt) node[below=3pt, color2] {$(C_1,  \ldots,  C_n,  M_1)$};
    \draw [fill=color2] (5,0) circle (4pt) node[below=3pt, color2] {$(D_1,  \ldots,  D_k,  M_2)$};
    \draw [fill=color2] (5,5) circle (4pt) node[above=3pt, color2] {$(D_1,  \ldots,  D_k,  M_3)$};
    \draw [fill=color2] (0,5) circle (4pt) node[above=3pt, color2] {$(C_1,  \ldots,  C_n,  M_0)$};
    
    \draw [fill=color2] (12,5) circle (4pt) node[above=3pt, color2] {$M_0$};
    \draw [fill=color2] (12,0) circle (4pt) node[below=3pt, color2] {$M_1$};
    \draw [fill=color2] (17,5) circle (4pt) node[above=3pt, color2] {$C_{\alpha(k)+1}\otimes \cdots \otimes C_n \otimes M_0$};
    \draw [fill=color2] (17,0) circle (4pt) node[below=3pt, color2] {$C_{\alpha(k)+1}\otimes \cdots \otimes C_n \otimes M_1$};
    \draw [fill=color2] (22,0) circle (4pt) node[below=3pt, color2] {$M_2$};
    \draw [fill=color2] (22,5) circle (4pt) node[above=3pt, color2] {$M_3$};
    \end{small}

    \fill[line width=2pt,color=color5,fill=color5,fill opacity=0.1] (26,2.5) -- (28,1.5) -- (29,3) -- (27.5,5) -- cycle;
    \fill[line width=2pt,color=color1,fill=color1,fill opacity=0.6] (24,1) -- (26,0) -- (28,1.5) -- (26,2.5) -- cycle;

    \begin{tiny}
    \draw [->,line width=1pt] (24,1) -- (26,0) node[below left=1pt,  midway] {$\alpha$};
    \draw [->,line width=1pt] (26,0) -- (28,1.5) node[below right=1pt,  midway] {$q(e)_!$};
    \draw [->,line width=1pt] (28,1.5) -- (29,3) node[below=1pt,  midway] {$\beta$};
    \draw [->,line width=1pt] (24,1) -- (26,2.5) node[above left=1pt,  midway] {$q(e)_!$};
    \draw [->,line width=1pt] (26,2.5) -- (28,1.5) node[below left=1pt,  midway] {$q(e)_! \alpha$};
    \draw [->,line width=1pt] (26,2.5) -- (27.5,5) node[above=1pt,  midway] {$\delta$};
    \draw [->,line width=1pt,dash pattern=on 10pt off 10pt] (26,2.5) -- (29,3) node[below=1pt,  midway] {$\gamma$};
    \draw [->,line width=1pt] (28,1.5) -- (27.5,5) node[above right=1pt,  midway] {$\eta$};
    \draw [->,line width=1pt] (29,3) -- (27.5,5) node[above=1pt,  midway] {$\theta$};
    
    \draw [fill=color2] (24,1) circle (4pt) node[left=3pt, color2] {$M_0$};
    \draw [fill=color2] (26,0) circle (4pt) node[below=3pt, color2] {$M_1$};
    \draw [fill=color2] (28,1.5) circle (4pt) node[below right=3pt, color2] {$C_{\alpha(k)+1}\otimes \cdots \otimes C_n \otimes M_1$};
    \draw [fill=color2] (29,3) circle (4pt) node[right=3pt, color2] {$M_2$};
    \draw [fill=color2] (26,2.5) circle (4pt) node[above=3pt, color2] {$C_{\alpha(k)+1}\otimes \cdots \otimes C_n \otimes M_0$};
    \draw [fill=color2] (27.5,5) circle (4pt) node[above=3pt, color2] {$M_3$};
    \end{tiny}
    
\end{tikzpicture}
}

\caption{Some examples of the phenomenon: on the left a simplex in $\mathcal{M}^{\varoast}\times_{\mathcal{C}^{\varoast}}\Delta^1$,  and on the right the corresponding diagram in $\mathcal{M}$. The blue rectangle on the right represents the copy of $\mathcal{M}$ over $0$ (or over $(C_1,\ldots, C_n)$), the green one on the left is the one over $1$ (or over $(D_1,\ldots, D_k)$)}
\label{fig:simplices.in.Delta1.times.M}
\end{figure}
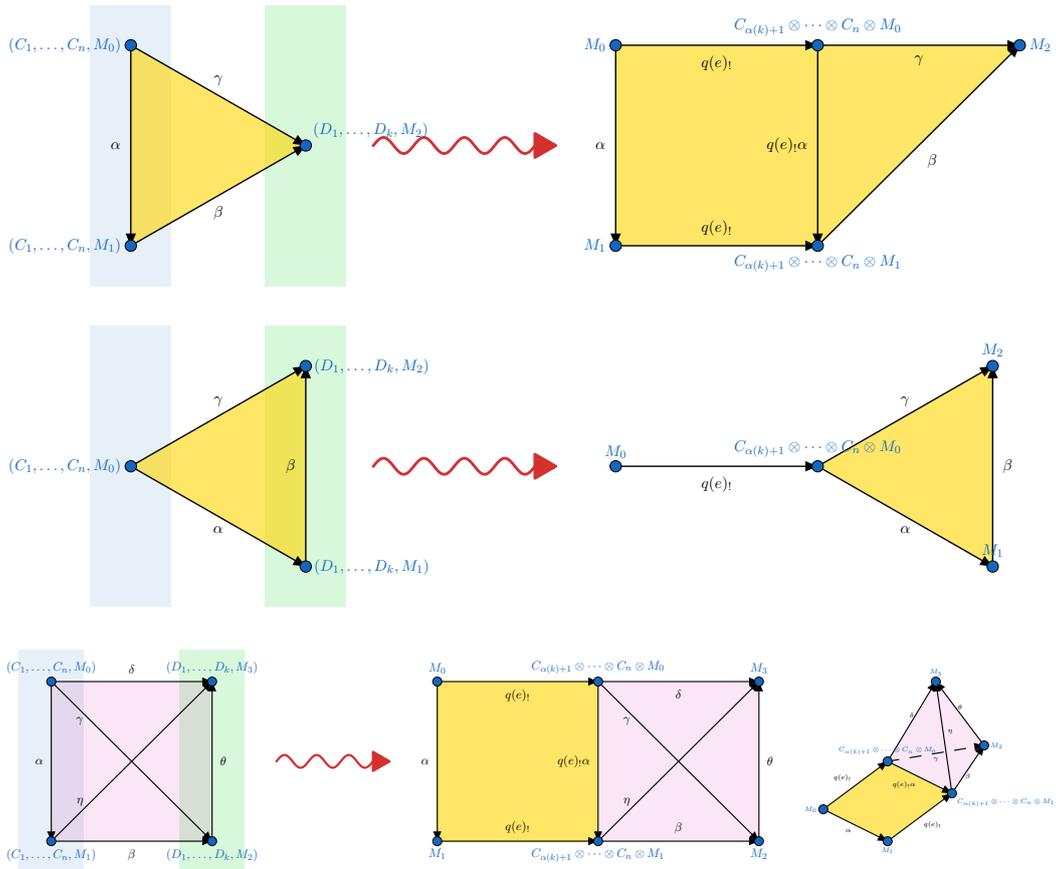

In fact, consider an $n$ simplex $\Delta^n \to \mathcal{M}^{\varoast}\times_{\mathcal{C}^{\varoast}} \Delta^1$; say that $m$ of its vertices (with $0 \leq m \leq n+1$) belong to the copy of $\mathcal{M}$ over $0 \in \Delta^1$. If either $m=0$ or $m=n+1$, then the simplex is just an $n$-simplex of $\mathcal{M}$. For other values of $m$, consider the pushout in simplicial sets:
\[\begin{tikzcd}
    \Delta^{\{0,\ldots,m-1\}} 
        \arrow[r, "i_1"] 
        \arrow[d]
        \arrow[dr, phantom, "\ulcorner", very near end]&
    \Delta^{m-1} \times \Delta^1 
        \arrow[d]\\
    \Delta^n 
        \arrow[r]&
    P_{n,m}
\end{tikzcd}\addtocounter{tikzcomd}{1}\]
In practice, the face $\Delta^{m-1} \times \{1\}$ of $\Delta^{m-1} \times \Delta^1$ is glued to $\Delta^{\{0,\ldots,m-1\}} \subset \Delta^n$. Having an $n$-simplex $\Delta^n \to \mathcal{M}^{\varoast}\times_{\mathcal{C}^{\varoast}} \Delta^1$ as above is then the same as having map $P_{n,m} \to \mathcal{M}$ such that: 
\begin{itemize}
    \item $\Delta^{m-1} \times \{0\} \to \mathcal{M}$ describes the $(m-1)$-simplex over $0 \in \Delta^1$,
    \item $\Delta^{m-1} \times \Delta^1 \to \mathcal{M}$ is the transport map $q(e)_!$ between the two fibres, and
    \item $\Delta^n \to \mathcal{M}$ describes the fibre over $1 \in \Delta^1$: this involves not only the zero-simplices over $1$, but also the images via $q(e)_!$ of the zero-simplices over $0$ and all the (higher) morphisms between them.
\end{itemize}

Hence all simplices in $\mathcal{M}^{\varoast}\times_{\mathcal{C}^{\varoast}} \Delta^1$ are given by certain (specific) diagrams in $\mathcal{M}$; the action of the category $\mathcal{C}^{\varoast}$ is reduced to the transport morphism $q(e)_!$, which is fixed for all of them.

Finally, observe that $m=0$ and $m=n+1$ can be seen as special cases of the general one, with a degenerate push-forward in the first, describing an $n$ simplex in the fibre over 1 only, and a $\Delta^n \times \Delta^1$ in the second, which corresponds to an $n$ simplex in the first fibre and its image via $q(e)_!$ in the second. The notation $P_{n,0}\to \mathcal{M}$ or $P_{n,n}\to \mathcal{M}$ is adopted for these cases.

To prove that $r(e)$ is locally $s$-co-Cartesian, construct any diagram:
\[\begin{tikzcd}
    &\Delta^1 
        \arrow[ldd, bend right=10]
        \arrow[d]
        \arrow[rd, "e"]&&&\\
    & \mathcal{M}^{\varoast}\times_{\mathcal{C}^{\varoast}} \Delta^1 
        \arrow[r]
        \arrow[d, "r'"]&
    \mathcal{M}^{\varoast}
        \arrow[d, "r"]&\\
    \Lambda^n_0
        \arrow[d]
        \arrow[r]&
    \mathcal{M}'^{\varoast} \times_{\mathcal{C}'^{\varoast}} \Delta^1
        \arrow[d, "\Tilde{s}"]
        \arrow[r]&
    \mathcal{M}'^{\varoast} \times_{\mathcal{C}'^{\varoast}}\mathcal{C}^{\varoast}
        \arrow[d, "s"]
        \arrow[r]
        \arrow[ul, phantom, "\lrcorner", very near end]&
    \mathcal{M}'^{\varoast}
        \arrow[d, "q'"]\\
    \Delta^n
        \arrow[r]
        \arrow[ur, dashed]&
    \Delta^1
        \arrow[r]&
    \mathcal{C}^{\varoast}
        \arrow[r, "f^{\varoast}"]
        \arrow[ul, phantom, "\lrcorner", very near end]&
    \mathcal{C}'^{\varoast}
        \arrow[ul, phantom, "\lrcorner", very near end]\\
\end{tikzcd}\addtocounter{tikzcomd}{1}\]
with similar requirements as above. To simplify the notation, the identities:
\begin{gather*}
    \mathcal{M}'^{\varoast} \times_{\mathcal{C}'^{\varoast}} \Delta^1 \simeq (\mathcal{M}'^{\varoast} \times_{\mathcal{C}'^{\varoast}}\mathcal{C}^{\varoast})\times_{\mathcal{C}^{\varoast}} \Delta^1\\
    \mathcal{M}^{\varoast}\times_{\mathcal{C}^{\varoast}} \Delta^1 \simeq \mathcal{M}^{\varoast}\times_{(\mathcal{M}'^{\varoast} \times_{\mathcal{C}'^{\varoast}}\mathcal{C}^{\varoast})} (\mathcal{M}'^{\varoast} \times_{\mathcal{C}'^{\varoast}} \Delta^1)
\end{gather*}
were used.

To find the desired lift, observe that $r'$ is an isomorphism of simplicial sets. In fact, the above conclusions about $\mathcal{M}^{\varoast}\times_{\mathcal{C}^{\varoast}} \Delta^1$ apply to $\mathcal{M}'^{\varoast}\times_{\mathcal{C}'^{\varoast}} \Delta^1$ as well. In particular, all simplices of the latter are described by diagrams $P_{n,m} \to \mathcal{M}'$, with the same requirements as above and the transfer map given by $q'(g(e))_!$. 

As $r'$ respects the distribution of the zero simplices in the fibres, it respects the indices $m$ and $n$. Consider an $n$ simplex $\Delta^n \to \mathcal{M}^{\varoast}\times_{\mathcal{C}^{\varoast}} \Delta^1$ given by a diagram $d: P_{n,m} \to \mathcal{M}$. Its image via $r'$ is an $n$ simplex $\Delta^n \to \mathcal{M}'^{\varoast}\times_{\mathcal{C}^{\varoast}} \Delta^1$ given by a diagram $P_{n,m} \to \mathcal{M}'$. Now, the map $r: \mathcal{M}^{\varoast} \to \mathcal{M}'^{\varoast}$ restricts to $g:\mathcal{M} \to \mathcal{M}'$ on the fibres over 0, and $q(e)_!$ is sent to $q'(g(e))_!$. As $g$ is an isomorphism of simplicial sets, for any pair of fixed $m,n$, this determines a bijection between the sets of such diagrams, hence $r'$ is an isomorphism.

So $\Lambda^n_0 \to \mathcal{M}'^{\varoast}\times_{\mathcal{C}'^{\varoast}} \Delta^1$ lifts to $\Lambda^n_0 \to \mathcal{M}^{\varoast}\times_{\mathcal{C}^{\varoast}} \Delta^1$. Hence there is a map $\Delta^n \to \mathcal{M}^{\varoast}$ making the diagram commute, and this in turn, via the universal property of the pullback, produces the desired $\Delta^n \to \mathcal{M}'^{\varoast}\times_{\mathcal{C}'^{\varoast}} \Delta^1$. So $r(e)$ results locally $s$ co-Cartesian.

Finally, let $Z$ be an object of $\mathcal{C}^{\varoast}$; via the canonical map $\mathcal{C}^{\varoast} \to N(\Delta)^{op}$, one can identify it with an n-tuple of objects of $\mathcal{C}$ \cite[Definition 4.1.3.2]{Lurie2017}; the fibre via $q$ of this n-tuple is a copy of $\mathcal{M}$ \cite[Remark 4.2.2.18]{Lurie2017}. The n-tuple in $\mathcal{C}^{\varoast}$ is sent via $f^{\varoast}$ to an n-tuple of objects of $\mathcal{D}^{\varoast}$; the fibre via $q'$ of this n-tuple is a copy of $\mathcal{M'}$. Hence $(\mathcal{M}'^{\varoast} \times_{\mathcal{C}'^{\varoast}}\mathcal{C}^{\varoast})_Z \simeq \mathcal{M}'$. By hypothesis $g$ is an equivalence $\mathcal{M} \to \mathcal{M'}$, so the third hypothesis is satisfied. Hence $\mathcal{M} \simeq \mathcal{M}'^{\varoast} \times_{\mathcal{C}'^{\varoast}}\mathcal{C}^{\varoast}$ by corollary \ref{cor:lurie.2.4.4.4}.

There is then an equivalence $\mathcal{M}^{\varoast} \times_{\mathcal{C}^{\varoast}} N(\Delta)^{op} \simeq \mathcal{M}'^{\varoast} \times_{\mathcal{C}'^{\varoast}} N(\Delta)^{op}$, where the pullbacks are taken along the maps appearing in diagram \ref{diag:LModA(M)}. This produces an equivalence:
\begin{multline}\label{eq:Fun=Fun.for.LModA}
    Fun_{N(\Delta)^{op}}(N(\Delta)^{op},\mathcal{M}^{\varoast} \times_{\mathcal{C}^{\varoast}} N(\Delta)^{op})  \\
    \simeq Fun_{N(\Delta)^{op}}(N(\Delta)^{op}, \mathcal{M}'^{\varoast} \times_{\mathcal{C}'^{\varoast}} N(\Delta)^{op})
\end{multline}

Now, the category $LMod^{\mathbb{A}_{\infty}}_A(\mathcal{M})$ is given \cite[Definition 4.2.2.10 and Remark 4.2.2.19]{Lurie2017} by the full subcategory on those functors that correspond to left modules in $Fun_{N(\Delta)^{op}}(N(\Delta)^{op},\mathcal{M}^{\varoast})$ (similarly for and $LMod^{\mathbb{A}_{\infty}}_{F(A)}(\mathcal{M}')$). This can be tested on two conditions. 
\begin{defi} \label{def:LMod(M)}
    Let $X \in Fun_{N(\Delta)^{op}}(N(\Delta)^{op},\mathcal{M}^{\varoast})$. Then $X \in LMod^{\mathbb{A}_{\infty}}(\mathcal{M})$ if:
\begin{enumerate}
    \item The functor induced by $X$ and by $q$ in $Fun_{N(\Delta)^{op}}(N(\Delta)^{op},\mathcal{C}^{\varoast})$ is in $Alg_{\mathbb{A}_{\infty}}(\mathcal{C}^{\varoast})$.
    \item Given $\alpha : [m] \to [n]$ an inert morphism in $\Delta$, such that $\alpha(m)=n$, then the induced map $X([n]) \to X([m])$ is a $p \circ q$-co-Cartesian morphism in $\mathcal{M}^{\varoast}$.
\end{enumerate}
\end{defi}

Now, if the arrow $X$ comes from an arrow $Y: N(\Delta)^{op} \to \mathcal{M}^{\varoast} \times_{\mathcal{C}^{\varoast}} N(\Delta)^{op}$, the first condition is clearly true, with the algebra given by $A$. Expanding the second condition, one gets the lifting problem:
\begin{equation}\label{diag:lifting.A}
\begin{tikzcd}
    & \Delta^{\{0,1\}}
        \arrow[rd, "\alpha"]
        \arrow[ldd]&\\
    &&N(\Delta)^{op} 
        \arrow[ld, swap, "Y"] 
        \arrow[d, "X"]\\
    \Lambda^n_0
        \arrow[dd]
        \arrow[rr, bend left=20]&
    \mathcal{M}^{\varoast} \times_{\mathcal{C}^{\varoast}} N(\Delta)^{op} 
        \arrow[r, "i"] 
        \arrow[d, "j"] & 
    \mathcal{M}^{\varoast}
        \arrow[d, "q"]\\
    &N(\Delta)^{op} 
        \arrow[rd, swap,"id"] 
        \arrow[r, "A"]  &
    \mathcal{C}^{\varoast} 
        \arrow[d, "p"]
        \arrow[ul, phantom, "\lrcorner", very near end]\\
    \Delta^n 
        \arrow[rr]
        \arrow[rruu, dashed, gray, bend right=10]&
    & N(\Delta)^{op}
\end{tikzcd}\addtocounter{tikzcomd}{1}
\end{equation}

This is equivalent to the lifting problem:
\begin{equation}\label{diag:lifting.B}
\begin{tikzcd}
    & \Delta^{\{0,1\}}
        \arrow[rd, "\alpha"]
        \arrow[ldd]&\\
    &&N(\Delta)^{op} 
        \arrow[ld, swap, "Y"] 
        \arrow[d, "X"]\\
    \Lambda^n_0
        \arrow[d]
        \arrow[r]&
    \mathcal{M}^{\varoast} \times_{\mathcal{C}^{\varoast}} N(\Delta)^{op} 
        \arrow[r, "i"] 
        \arrow[d, "j"] & 
    \mathcal{M}^{\varoast}
        \arrow[d, "q"]\\
    \Delta^n 
        \arrow[r]
        \arrow[ru, dashed, gray]&
    N(\Delta)^{op} 
        \arrow[rd, swap,"id"] 
        \arrow[r, "A"]  &
    \mathcal{C}^{\varoast} 
        \arrow[d, "p"]
        \arrow[ul, phantom, "\lrcorner", very near end]\\
    && N(\Delta)^{op}
\end{tikzcd}\addtocounter{tikzcomd}{1}
\end{equation}
In fact, first notice that the map $\Delta^n \to N(\Delta)^{op}$ can equivalently have the image in either of the $N(\Delta)^{op}$ appearing in the two lowest rows of the diagram, as they are connected by the identity morphism. 

Suppose now that one can solve all lifting problems of the form \ref{diag:lifting.A} and is presented with a lifting problem like \ref{diag:lifting.B}. Then, by composition, an arrow $\Lambda_0^n \to \mathcal{M}^{\varoast}$ is obtained, which leads to a lifting problem of the form \ref{diag:lifting.A}. This can be solved by assumption, so one gets a lift $\Delta^n \to \mathcal{M}^{\varoast}$. The lift $\Delta^n \to \mathcal{M}^{\varoast} \times_{\mathcal{C}^{\varoast}}N(\Delta)^{op}$ is finally given by the universal property of the pullback.

Next, suppose that one can solve all lifting problems of the form \ref{diag:lifting.B} and is presented with a diagram \ref{diag:lifting.A}. One first passes to a diagram of the form \ref{diag:lifting.B} by pulling back the map $\Lambda^n_0 \to \mathcal{M}^{\varoast}$ along the composition $\Lambda^n_0 \to \Delta^n \to N(\Delta)^{op}$, to get a map $\Lambda^n_0 \to \mathcal{M}^{\varoast} \times_{\mathcal{C}^{\varoast}}N(\Delta)^{op}$. This lifting problem can be solved by assumption, so one gets a map $\Delta^n \to \mathcal{M}^{\varoast} \times_{\mathcal{C}^{\varoast}}N(\Delta)^{op}$; one finally gets the desired map $\Delta^n \to \mathcal{M}^{\varoast}$ by composition.

In case $X= i \circ Y$ comes from a module over a specified algebra $A$, it is then possible to reformulate condition $2.$ in \ref{def:LMod(M)} as:
\begin{itemize}
    \item [2'.] Given $\alpha : [m] \to [n]$ an inert morphism in $\Delta$, such that $\alpha(m)=n$, then the induced map $Y([n]) \to Y([m])$ is a $j$-co-Cartesian morphism in $\mathcal{M}^{\varoast}\times_{\mathcal{C}^{\varoast}}N(\Delta)^{op}$.
\end{itemize}
Given that the equivalence $\mathcal{M}^{\varoast}\times_{\mathcal{C}^{\varoast}}N(\Delta)^{op} \simeq \mathcal{M}'^{\varoast}\times_{\mathcal{C}'^{\varoast}}N(\Delta)^{op}$ comes equipped, by construction, with a commutative diagram:
\[\begin{tikzcd}
    \mathcal{M}^{\varoast}\times_{\mathcal{C}^{\varoast}}N(\Delta)^{op} 
        \arrow[rr, "\sim"]
        \arrow[rd, swap, "j"]&&
    \mathcal{M}'^{\varoast}\times_{\mathcal{C}'^{\varoast}}N(\Delta)^{op}
        \arrow[ld, "j'"]\\
    & N(\Delta)^{op}&
\end{tikzcd}\addtocounter{tikzcomd}{1}\]
where $j$ and $j'$ are the maps induced by the pullbacks, it is clear that an arrow in $\mathcal{M}^{\varoast}\times_{\mathcal{C}^{\varoast}}N(\Delta)^{op}$ is $j$-co-Cartesian if and only if the corresponding arrow is $j'$-co-Cartesian. Hence the equivalence \ref{eq:Fun=Fun.for.LModA} restricts to an equivalence:
\[
    LMod_A(\mathcal{M}^{\varoast}) \simeq LMod_{F(A)}(\mathcal{M}'^{\varoast})
\]
\end{proof}

\clearpage

\section{The enriched adjunction}\label{sec:enr.adj}

In this section, the results of the previous part are applied to operads and monads, providing a link between operads enriched in a certain category and monads over the same category.

Given $\mathcal{C}$ any $\infty$-category, let
\[
    sSeq(\mathcal{C})=Fun \left(\coprod_{n\geq 0} B(\Sigma_n), \mathcal{C}\right)
\]
be the category of symmetric sequences in $\mathcal{C}$ \cite[Notation 8.1]{Heine2023a}. In the event $\mathcal{C}$ is endowed with a symmetric monoidal structure compatible with small colimits, one can define a monoidal structure on symmetric sequences as well, called the composition product. Given $X,\,Y \in sSeq(\mathcal{C})$, their composition product is described level-wise by the following formula \cite[Proposition 8.2]{Heine2023a}:
\[
    (X \circ Y)(n) \cong \coprod_{k \geq 0} \left ( \coprod_{n_1 + \ldots + n_k =n} \Sigma_n \times_{(\Sigma_{n_1} \times \ldots \times \Sigma_{n_k})}(X_k \otimes (\bigotimes_{1 \leq j \leq k} Y_{n_j}))\right)_{\Sigma_k}
\]
Observe in particular that if $Y$ is concentrated in degree zero, so is $X \circ Y$, for any symmetric sequence $X$. Identifying $\mathcal{C}$ with the symmetric sequences concentrated in degree zero provides then a left action of $sSeq(\mathcal{C})$ on $\mathcal{C}$.

\begin{defi}
    A single coloured operad $O$ enriched in $\mathcal{C}^{\otimes}$ is an associative algebra in $sSeq(\mathcal{C})$. The category of $O$-algebras $Alg_O(\mathcal{C})$ in $\mathcal{C}$ is defined as the category of left $O$-modules with respect to the action of $sSeq(\mathcal{C})$ on $\mathcal{C}$: $Alg_O(\mathcal{C})=LMod_O(\mathcal{C})$.
\end{defi} 

This approach to $\infty$-operads is very similar to the original, classical one by Kriz and May \cite{KriMay1995} and represents one of the most common ways to deal with enriched $\infty$-operads. For a more thorough study of the construction,  we refer to \cite{Haugseng2019}, where one can also find a multicoloured version.
This definition of $\infty$-operads via symmetric sequences has also been confronted with the other formulations; 
for instance, the construction of $\infty$-operads in the category of spaces is compared with the more common definition of $\infty$-operads by Lurie \cite[Definition 2.1.1.10]{Lurie2017} in \cite[Remark 5.15]{Haugseng2019}, proving a partial equivalence between the two constructions.

Next, we associate to an $\infty$-operad in $\mathcal{C}^{\otimes}$ a monad on $\mathcal{C}$. In fact, one can always interpret the action of a symmetric sequence in $\mathcal{C}$ as an endofunctor $\mathcal{C} \to \mathcal{C}$, which promotes to a monoidal functor from the category $sSeq(\mathcal{C})^{\otimes}$ of symmetric sequences in $\mathcal{C}$ under composition product to the category $Fun(\mathcal{C},\mathcal{C})^{\otimes}$ of endofunctors of $\mathcal{C}$ under composition. 

A monad is an associative algebra object $T \in Alg_{Assoc}Fun(\mathcal{C},\mathcal{C})$ with respect to the composition monoidal structure \cite[Definition 4.7.0.1]{Lurie2017}. The monoidal functor from symmetric sequences to endofunctors introduced above produces hence a map:
\[
    Alg_{Assoc}(sSeq(\mathcal{C})) \to Alg_{Assoc}(End_{\mathcal{C}}(\mathcal{C}))
\]
sending a $\mathcal{C}$-enriched operad $O$ to the corresponding 
monad ${T_O}$.

As symmetric sequences and endofunctors act on the left on $\mathcal{C}$, compatibly with the monoidal functor between them, we can apply the results of the previous section: by theorem \ref{thm:LModA=LModf(A)}, 
\begin{equation}\label{eq:algO(C)=LMod{T_O}(C)}
Alg_O(\mathcal{C})\simeq LMod_{T_O}(\mathcal{C}).    
\end{equation}

Now, it is in fact possible to put additional structure on monads and the associated module categories. Monads can in fact be defined in any $(\infty,2)$-category:

\begin{defi}{\cite[Definition 4.1]{Heine2023a}}
    Let $\mathcal{D}^{\circledast} \to Cat^{\times}_{\infty}$ be an $(\infty,2)$-category, and let $X$ be an object of $\mathcal{D}$. A monad on $X$ in $\mathcal{D}$ is an associative algebra in the endomorphism monoidal structure on $Mor_{\mathcal{D}}(X,X)$.
\end{defi}

The previous definition of monad corresponds to the choice $\mathcal{D}=Cat_{\infty}$ and is considered the non-enriched case.
In the enriched case the category of left modules over a monad $T$ on an object $X \in \mathcal{D}$ is formally substituted with that of Eilenberg-Moore object \cite[Definition 5.2.]{Heine2023a}, encoded as a functor $Y \to X$ in $\mathcal{D}$ satisfying certain universal properties. Observe that in the non-enriched case, this functor $Y \to X$ would simply be the forgetful functor $LMod_T(X) \to X$ in $Cat_{\infty}$. This fact is somewhat reflected by the functoriality of Eilenberg-Moore objects: \cite[Corollary 5.16]{Heine2023a} states that 2-functors that admit 2-left adjoints preserve Eilenberg-Moore objects.

Eilenberg-Moore objects do not always exist but appear quite often in nature. For instance if one deals with $\mathcal{D}$ presentably left tensored over $Cat_{\infty}$ (for example $\mathcal{D}=Cat_{\infty}^{\mathcal{V}}$ the $\mathcal{V}$-enriched $\infty$ categories, with $\mathcal{V}^{\otimes}$ presentably symmetric monoidal), one has all Eilenberg-Moore objects  \cite[Corollary 5.28]{Heine2023a}. However, in our case, the following holds:

\begin{prop}\label{prop:opd.to.monad}
    Let $\mathcal{C}$ be a symmetric monoidal $\infty$-category; suppose that $\mathcal{C}$ is enriched over another $\infty$-category $\mathcal{V}^{\otimes}$, which carries the cartesian monoidal structure. Consider any $\mathcal{C}$-enriched infinity operad $O \in Alg_{Assoc}(sSeq(\mathcal{C})$; then the associated monad ${T_O}$ promotes to a lax $\mathcal{V}^{\otimes}$-linear monad on $\mathcal{C}$.
\end{prop}

\begin{proof}
    To appear in \cite{Heine2023a}.
\end{proof}

\begin{cor}\label{cor:modules.are.enriched}
    There is an Eilenberg-Moore object for the lax $\mathcal{V}^{\otimes}$-linear monad monad ${T_O}$; in other words, the $\infty$-category of left modules $LMod_{T_O}(\mathcal{C})$ can be promoted to a weakly left $\mathcal{V}^{\otimes}$-tensored $\infty$-category.
\end{cor}
\begin{proof}
    This is the content of \cite[Corollary 5.36]{Heine2023a}.
\end{proof}

From now on, we will always promote the usual notation for the non-enriched context to the enriched one. Namely, we will refer to the lax $\mathcal{V}^{\otimes}$-linear monad as ${T_O}$ and to the weakly $\mathcal{V}^{\otimes}$ enriched category of left modules as $LMod_{T_O}(\mathcal{C})$. Moreover, we will refer to this weakly enriched version of $LMod_{T_O}(\mathcal{C})$ also as $Alg_O(\mathcal{C})$, whenever we want to stress that the monad comes from an operad, given the equivalence \ref{eq:algO(C)=LMod{T_O}(C)} in the non enriched context. These choices are coherent with the fact that, after forgetting the enrichment, we obtain the usual (spaces-enriched) categories.

We proceed now to the tensoring result. 

\begin{hypothesis}\label{hyp:V}
    Let $\mathcal{V}$ be an $\infty$-category that admits finite products and small colimits, such that the product preserves small colimits independently in each variable. We view $\mathcal{V}$ as a symmetric monoidal $\infty$-category $\mathcal{V}^{\otimes}$ via cartesian product. 
\end{hypothesis}

\begin{rmk}\label{rmk:tensor.1.is.final}
    Since the monoidal structure in $\mathcal{V}$ is given by cartesian product, tensor units and terminal objects coincide. We indicate any of them with the symbol $\mathbf{1}_{\mathcal{V}}$.
\end{rmk}

\begin{hypothesis} \label{hyp:C}
    Let $\mathcal{C}^{\otimes}$ be a symmetric monoidal $\infty$-category compatible with small colimits and let $\phi: \mathcal{V}^{\otimes} \to \mathcal{C}^{\otimes}$ be a symmetric monoidal functor preserving small colimits.
\end{hypothesis}
In this setting, the induced left $\mathcal{V}^{\otimes}$-action on ${\mathcal{C}}$ is tensored, in other words, there is a well-defined tensor functor:
\[
    (-) \otimes (-): \mathcal{V} \times \mathcal{C} \xrightarrow{\phi \times id} {\mathcal{C}} \times {\mathcal{C}} \to {\mathcal{C}}
\]
Recall that the action is $\mathcal{V}$-enriched if, for any $X \in \mathcal{C}$, there is a right adjoint:
\[
    Mor_{\mathcal{C}} (X,-): \mathcal{C} \to \mathcal{V}
\]
to the functor:
\[
    (-) \otimes X : \mathcal{V} \to \mathcal{C}.
\]
We call $Mor_{\mathcal{C}} (-,-): \mathcal{C}\times \mathcal{C} \to \mathcal{V}$ a  morphism object for $\mathcal{C}$ in $\mathcal{V}$ \cite[Definition 3.16]{Heine2023b}. 

\begin{prop}\label{prop:tensor.algebras}
    In the context of hypothesis \ref{hyp:V} and \ref{hyp:C}, assume that the $\mathcal{V}$-action on $\mathcal{C}$ is also enriched. Let $O^{\otimes}$ be an $\infty$-operad in $\mathcal{C}$. Then the $\infty$-category $Alg_{O}({\mathcal{C}})$ is left tensored and enriched over $\mathcal{V}$.
\end{prop}

\begin{proof}
    We saw that via proposition \ref{prop:opd.to.monad} we can associate to the operad $O$ in $\mathcal{C}$ a lax $\mathcal{V}$-linear monad ${T_O}$, so that the category $Alg_{O}(\mathcal{C})=LMod_{T_O}(\mathcal{C})$ results weakly left tensored.
    
    We can then apply \cite[Proposition 5.41]{Heine2023a}; all the hypothesis of points 1,2,3 are verified, so we conclude that $LMod_{T_O}(\mathcal{C})\cong Alg_O(\mathcal{C})$ is left tensored and enriched over $\mathcal{V}$.

\end{proof}

\begin{rmk}\label{rmk:the.first.adjoint}
     As, for all $C \in Alg_{O}({\mathcal{C}})$, the functor:
        \[
            (-) \otimes C: \mathcal{V} \to Alg_{O}({\mathcal{C}})
        \]
        admits a right adjoint, the internal morphism object:
        \[
            Mor(C,-): Alg_{O}({\mathcal{C}}) \to \mathcal{V},
        \]
        the tensor product preserves colimits in the first variable by \cite[Proposition 5.2.3.5.]{Lurie2009}. Observe in particular that this does not require any presentability assumption on the categories involved.
    \end{rmk}

We wish to show a similar result for cotensors as well: we would like to prove that, under suitable assumptions, if the base category $\mathcal{C}$ admits cotensors for the action of $\mathcal{V}$, then also the algebra category $Alg_O(\mathcal{C}$ does. In particular, we begin by showing that, in an appropriate context (see proposition \ref{prop:tensoring.preserves.colimits}), given $V \in \mathcal{V}$, the functor
\[
    V \otimes (-): LMod_T(\mathcal{C}) \to LMod_T(\mathcal{C})
\]
preserves all small colimits. Observe incidentally that \cite[Theorem 7.3]{Heine2023a} already deals with geometric realisations. Hence, we start by proving that a functor between suitable $\infty$-categories preserves small colimits if and only if it preserves small coproducts and geometric realisations.

\begin{prop}
   Let $\mathcal{A}$ be an $\infty$-category. Then $\mathcal{A}$ is closed under small colimits if and only if it is closed under small sifted colimits and finite coproducts.
\end{prop}

\begin{proof}
By \cite[Corollary 4.2.3.11]{Lurie2009} an infinity category $\mathcal{A}$ admits small colimits if and only if it admits finite colimits and colimits indexed by the nerves of small filtered partially ordered sets; hence $\mathcal{A}$ admits small colimits if and only if it admits finite colimits and small filtered colimits.

By \cite[Lemma 1.3.3.10]{Lurie2017} $\mathcal{A}$ admits finite colimits if it has finite coproducts and geometric realisations of simplicial objects; now filtered colimits and geometric realisations of simplicial objects are sifted respectively by \cite[Example 5.5.8.3]{Lurie2009}  and \cite[Example 5.5.8.4]{Lurie2009}. Hence an infinity category $\mathcal{A}$ is closed under small colimits if and only if it is closed under finite coproducts and small sifted colimits.
\end{proof}

\begin{lemma}\label{lemma:{T_O}.preserves.geom.real}
    If $\mathcal{C}^{\otimes}$ is a symmetric monoidal category such that the tensor product is compatible with small colimits component-wise and $T={T_O}$ is a monad coming from an operad $O \in Alg(sSeq(\mathcal{C}))$, then ${T_O}$ preserves geometric realisations.
\end{lemma}

\begin{proof}
Recall that by geometric realisation we mean the colimit of a simplicial object, i.e. of a functor:
\[
    N(\Delta^{op}) \to \mathcal{C}
\]
The indexing category $N(\Delta^{op})$ is sifted \cite[Lemma 5.5.8.4.]{Lurie2009}. By \cite[Proposition 3.8]{Haugseng2019}, the functor ${T_O}:\mathcal{C} \to \mathcal{C}$ preserves sifted colimits.
\end{proof}

\begin{cor}\label{cor:AlgoC.has.col}
    In the above context, $Alg_O(\mathcal{C})$ is closed under small colimits.
\end{cor}

\begin{proof}
    By \ref{lemma:{T_O}.preserves.geom.real}, the monad associated to the operad $O$ preserves geometric realisations; it is then shown in the proof of \cite[Theorem 7.3]{Heine2023a} that $Alg_O(\mathcal{C})$ (there appearing as $LMod_T(\mathcal{C})$; we are always working under the previously discussed equivalence) has small coproducts. In \cite[Remark 2.12]{Haugseng2019} it is shown that it has all small (limits and) sifted colimits, which are detected by the forgetful functor. Hence we can apply the previous proposition.
\end{proof}

What follows is an adaptation of \cite[Proposition 5.5.8.15]{Lurie2009} and \cite[Corollary 5.5.8.17]{Lurie2009}, which deal with sifted colimits, to colimits along any small diagram.

\begin{prop}
    Let $\mathcal{I}$ and $\mathcal{A}$ be small categories, with $\mathcal{A}$ admitting small coproducts and geometric realisations. Let $Fun_{\Omega}(\mathcal{P}(\mathcal{I}),\mathcal{A})\subseteq Fun(\mathcal{P}(\mathcal{I}),\mathcal{A})$ be the subcategory of functors preserving small coproducts and geometric realisations. Then:
    \begin{enumerate}
        \item Composition with the Yoneda embedding $j: \mathcal{I} \to\mathcal{P}(\mathcal{I})$ induces and equivalence:
        \[
            Fun_{\Omega}(\mathcal{P}(\mathcal{I}),\mathcal{A}) \to Fun(\mathcal{I},\mathcal{A})
        \]
        \item Any functor $g \in Fun_{\Omega}(\mathcal{P}(\mathcal{I}),\mathcal{A})$ preserves all small colimits.
    \end{enumerate}
\end{prop}
\begin{proof}
By \cite[Lemma 5.5.8.13]{Lurie2009} $\mathcal{P}(\mathcal{I})$ is the smallest full subcategory of $\mathcal{P}(\mathcal{I})$ containing the essential image of the Yoneda embedding and closed under small coproducts and geometric realisations. Hence, by \cite[Remark 5.3.5.9]{Lurie2009} and \cite[Proposition 4.3.2.15]{Lurie2009} we have the equivalence stated in the first claim.

To show that any $g\in Fun_{\Omega}(\mathcal{P}(\mathcal{I}),\mathcal{A})$ preserves all small colimits, we first notice, as in the proof of \cite[Proposition 5.5.8.15]{Lurie2009}, that $g$ preserves colimits if and only if $e \circ g$ does, for any $e: \mathcal{A} \to \mathcal{S}^{op}$ representable. Hence we may as well suppose $\mathcal{A}$ closed under colimits. Let then $Fun^L_{\Omega}(\mathcal{P}(\mathcal{I}),\mathcal{A})$ be the full subcategory of colimit preserving functors in $Fun_{\Omega}(\mathcal{P}(\mathcal{I}),\mathcal{A})$. As $\mathcal{P}(\mathcal{I})$ is the smallest full subcategory of $\mathcal{P}(\mathcal{I})$ containing the essential image of the Yoneda embedding and closed under small colimits, \cite[Remark 5.3.5.9]{Lurie2009} again provides an equivalence:
\[
    Fun^L_{\Omega}(\mathcal{P}(\mathcal{I}),\mathcal{A}) \cong Fun(\mathcal{I}, \mathcal{A})
\]
Hence the inclusion:
\[
    Fun^L_{\Omega}(\mathcal{P}(\mathcal{I}),\mathcal{A}) \subseteq Fun_{\Omega}(\mathcal{P}(\mathcal{I}),\mathcal{A})
\]
is an equivalence and therefore an equality, proving the thesis.
\end{proof}

\begin{prop}\label{prop:col.iff.geor.and.copr}
    Let $\mathcal{A}$ and $\mathcal{B}$ be $\infty$-categories, with $\mathcal{A}$  and $\mathcal{B}$ closed under small coproducts and geometric realisations. Then a functor $F: \mathcal{A} \to \mathcal{B}$ preserves small colimits if and only if it preserves small coproducts and geometric realisations.
\end{prop}

\begin{proof}
    One direction is clear. On the other hand, suppose $F$ preserves small coproducts and geometric realisations. Let $\bar{i}:\mathcal{I}^{\vartriangleright} \to \mathcal{A}$ be a colimit diagram; we wish to show that $F\circ \bar{p}$ is as well.

    Let $i=\bar{i}_{|\mathcal{I}}:\mathcal{I} \to \mathcal{A}$; by the previous proposition, $i$ is homotopic via left Kan extension to the composite:
    \[
        \mathcal{I} \xrightarrow{j} \mathcal{P}(\mathcal{I}) \xrightarrow{q} \mathcal{A}
    \]
    with $j$ the Yoneda embedding and $q$ preserving small colimits.

    Let $\bar{j}: \mathcal{I}^{\vartriangleright} \to \mathcal{P}(\mathcal{I})$ be a colimit for $j$ (which exists because $\mathcal{P}(\mathcal{I}) $ is opportunely closed under colimits). As $q$ preserves colimits, $q \circ \bar{j}$ is a colimit for $q \circ j \cong i$, hence $q \circ \bar{j} \cong \bar{i}$. To get the thesis, we can then equivalently verify that $F\circ q\circ \bar{j}$ is a colimit, which follows once we show that $F \circ q: \mathcal{P}(\mathcal{I}) \to \mathcal{B}$ preserves colimits. But $q$ preserves small colimits and $F$ preserves small coproducts and geometric realisations: again by the previous result, $F \circ q$ preserves small colimits.
\end{proof}

\begin{prop} \label{prop:tensoring.preserves.colimits}
    Let $V^{\otimes}$ be a monoidal category compatible with geometric realisations, and let $\mathcal{C}$ be a left $\mathcal{V}^{\otimes}$-tensored category, compatible with geometric realisations.
    Let $T$ be a lax $\mathcal{V}$-linear monad on $\mathcal{C}$ that preserves geometric realisations. Suppose that $LMod_T(\mathcal{C})$ is closed under colimits.
    
    Observe that under this assumptions \cite[Prop 5.41.(3)]{Heine2023a} holds, in other words, $LMod_T(\mathcal{C})$ is tensored on $\mathcal{V}$. Suppose that given $V \in \mathcal{V}$, the functor:
    \[
        V \otimes (-): \mathcal{C} \to \mathcal{C}
    \]
    preserves small colimits. Then the induced functor:
    \[
        V \otimes (-): LMod_T(\mathcal{C}) \to LMod_T(\mathcal{C})
    \]
    does as well.
\end{prop}

\begin{proof}
To make the notation clearer, we distinguish in this proof between the symbols of the tensor actions. In particular, they are denoted as:
\begin{gather*}
    V \otimes (-): \mathcal{C} \to \mathcal{C} \\
    V \tilde{\otimes} (-): LMod_T(\mathcal{C}) \to LMod_T(\mathcal{C}).
\end{gather*}

As $LMod_T(\mathcal{C})$ is closed under colimits, we can apply proposition \ref{prop:col.iff.geor.and.copr}: a functor $LMod_T(\mathcal{C})\to LMod_T(\mathcal{C})$ preserves small colimits if and only if it preserves small coproducts and geometric realisations. It was proven in \cite[Corollary 7.5]{Heine2023a} that under our assumptions
\[
V \tilde{\otimes} (-): LMod_T(\mathcal{C}) \to LMod_T(\mathcal{C})
\]
preserves geometric realisations. It is then enough to show that it preserves small coproducts.

By \cite[Proposition 5.41.3]{Heine2023a}, the free functor $T: \mathcal{C} \to LMod_T(\mathcal{C})$ preserves the action, so we have a commutative square:

\[
\begin{tikzcd}
    \mathcal{C} 
        \arrow[r, "V \otimes (-)"]
        \arrow[d, bend right=15, swap, "T"]
    &\mathcal{C}
        \arrow[d, bend right=15, swap, "T"]
    \\LMod_T(\mathcal{C})
        \arrow[r, "V \tilde{\otimes} (-)"]
        \arrow[u,  bend right=15, swap, "U"]
    &LMod_T(\mathcal{C})
        \arrow[u, bend right=15, swap, "U"]
\end{tikzcd}\addtocounter{tikzcomd}{1}
\]
where $U: LMod_T(\mathcal{C}) \to \mathcal{C} $ indicates the forgetful functor, right adjoint to the free monad functor.

Let $M=\coprod_I M_i \in LMod_T(\mathcal{C})$ be a coproduct of $T$-modules. By \cite[Example 4.7.2.7.]{Lurie2017}, each $M_i$ can be expressed as the geometric realisation of a simplicial object with free levels:
\[
    M_i=\colim_{j \in \Delta^{op}} T(M_{i,j})
\]
So:
\begin{align*}
V \tilde{\otimes} \coprod_{i\in I} M_i & \cong  V \tilde{\otimes} \coprod_{i\in I} \colim_{j \in \Delta^{op}} T(M_{i,j}) 
\intertext{We can commute the coproduct and the geometric realisation because the former has trivial transition maps:}
& \cong V \tilde{\otimes}  \colim_{j \in \Delta^{op}} \coprod_{i\in I} T(M_{i,j}) 
\intertext{By \cite[Corollary 7.5]{Heine2023a}, the tensoring $\tilde{\otimes}$ induced on $LMod_T(\mathcal{C})$ is compatible with geometric realisations:}
& \cong \colim_{j \in \Delta^{op}} ( V \tilde{\otimes}   \coprod_{i\in I} T(M_{i,j}))  
\intertext{As the free functor $T: \mathcal{C}\to LMod_T(\mathcal{C})$ is a left adjoint, it preserves colimits:}
& \cong \colim_{j \in \Delta^{op}} ( V \tilde{\otimes}   T(\coprod_{i\in I} M_{i,j}))  
\intertext{By \cite[Proposition 5.41.(3)]{Heine2023a}, the free functor preserves the tensor action:}
& \cong \colim_{j \in \Delta^{op}} T( V \otimes   \coprod_{i\in I} M_{i,j}))  
\intertext{By hypothesis, $V \otimes (-): \mathcal{C} \to \mathcal{C}$ preserves small colimits, in particular, coproducts:}
& \cong \colim_{j \in \Delta^{op}}  T( \coprod_{i\in I}(V \otimes   M_{i,j}))  
\intertext{Again since $T: \mathcal{C}\to LMod_T(\mathcal{C})$ is a left adjoint:}
& \cong \colim_{j \in \Delta^{op}} \coprod_{i\in I} T( V \otimes   M_{i,j})  
\intertext{Again because the transition maps in a coproduct are trivial:}
& \cong \coprod_{i\in I} \colim_{j \in \Delta^{op}}  T( V \otimes   M_{i,j})  
\intertext{Again because of \cite[Proposition 5.41.(3)]{Heine2023a}:}
& \cong \coprod_{i\in I} \colim_{j \in \Delta^{op}}   V \tilde{\otimes}  T( M_{i,j})  
\intertext{Again because \cite[Corollary 7.5]{Heine2023a} $\tilde{\otimes}$ is compatible with geometric realisations:}
& \cong \coprod_{i\in I}   V \tilde{\otimes} \colim_{j \in \Delta^{op}}  T( M_{i,j}) \cong \coprod_{i\in I} V \tilde{\otimes} M_i.
\end{align*}
As we wanted to show.
\end{proof}

\begin{cor}
    Let $V^{\otimes}$ be a monoidal category compatible with geometric realisations, and let $\mathcal{C}$ be a left $\mathcal{V}^{\otimes}$-tensored category, compatible with geometric realisations. Let $O$ be an operad in $\mathcal{C}$. Suppose that given $V \in \mathcal{V}$, the functor:
    \[
        V \otimes (-): \mathcal{C} \to \mathcal{C}
    \]
    preserves small colimits. Then the induced functor:
    \[
        V \otimes (-): Alg_O(\mathcal{C}) \to Alg_O(\mathcal{C})
    \]
    does as well.
\end{cor}

\begin{proof}
    Corollary \ref{cor:AlgoC.has.col} states that in this context $Alg_O(\mathcal{C})$ admits all colimits. We can then apply the previous theorem to the monad $T={T_O}$ (which is lax $\mathcal{V}^{\otimes}$-linear after \ref{prop:opd.to.monad}).
\end{proof}

\begin{cor} \label{cor:cotensors.on.algo(C)}
    Assume $\mathcal{C}^{\otimes}$ is a presentably symmetric monoidal category, left tensored on a symmetric monoidal category $\mathcal{V}$ compatible with geometric realisations. Let $O \in Alg(sSeq(\mathcal{C}))$ be an operad in $\mathcal{C}$.
    Let $V \in \mathcal{V}$ be such that the functor:
    \[
        V \otimes (-): \mathcal{C} \to \mathcal{C}
    \]
    preserves small colimits. Then there exists a cotensor functor:
    \[
        (-)^V: Alg_O(\mathcal{C}) \to Alg_O(\mathcal{C}),
    \]
    right adjoint to the tensoring action:
    \[
        V \otimes (-): Alg_O(\mathcal{C}) \to Alg_O(\mathcal{C}).
    \]
\end{cor}

\begin{proof}
    Recall that $\mathcal{C}^{\otimes}$ is a presentably symmetric monoidal infinity category if the underlying category is presentable and the tensor product preserves colimits independently in both variables. In these conditions, $Alg_O(\mathcal{C})$ is presentable as well \cite[Corollary 3.9]{Haugseng2019}.

    By the last corollary \ref{cor:cotensors.on.algo(C)}, the functor:
    \[
        V \otimes (-): Alg_O(\mathcal{C}) \to Alg_O(\mathcal{C})
    \]
    preserves small colimits, so, by the adjoint functor theorem \cite[Corollary 5.5.2.9]{Lurie2009} it admits a right adjoint.
\end{proof}

\begin{rmk}
    For this last part, we restricted ourselves to monads coming from operads; in fact, it is not clear what properties should a general monad $T$ on a presentable $\infty$-category $\mathcal{C}$ possess, to produce a presentable $\infty$-category of left modules $LMod_T(\mathcal{C})$.

    Some results in this direction that might be worth mentioning are \cite[Theorem 2.78]{AdaRos1994}, which deals with accessible monads in the context of classical categories, and \cite[Proposition 4.2.3.7]{Lurie2017}, which requires a colimits-preserving action.
\end{rmk}

\clearpage

\section{Tensoring against ``idempotent'' algebras}\label{sec:idemp.enr}

We begin this section with an observation on tensor units.

\begin{rmk}\label{rmk:unit}
The theory of monoidal $\infty$ categories is intrinsically unital: given a co-Cartesian fibration of $\infty$-operads $\mathcal{V}^{\otimes} \to Assoc^{\otimes}$, the unit corresponds to the morphism $\langle 0 \rangle \to \langle 1 \rangle$ in $Assoc$ \cite[Remark 4.1.1.12.]{Lurie2017}.

This behaviour propagates to the world of left tensored $\infty$-categories. Let $\mathcal{O}^{\otimes} \to \mathcal{LM}^{\otimes}$ be a co-Cartesian fibration of $\infty$-operads, and call $\mathcal{V}^{\otimes}$ the fibre over $Assoc^{\otimes} \subset \mathcal{LM}^{\otimes}$ and $\mathcal{C}$ the fibre over $(\langle 1\rangle, \langle 1\rangle^{\circ} )$, so that $\mathcal{C}$ is left tensored on $\mathcal{V}^{\otimes}$. Then a unit $\mathbf{1}_V$ of $\mathcal{V}$,   given by the fibre on $(\langle 0 \rangle, \varnothing) \in \mathcal{LM}$, provides a unit for the tensoring as well.
\end{rmk}

For the main result of this section, we need that the transferred $\mathcal{V}$-action on the category of algebras $Alg_O(\mathcal{C})$ is tensored and cotensored, so our hypothesis on the categories $\mathcal{C}$ and $\mathcal{V}^{\otimes}$ must answer both features.

Let $\mathcal{V}$ be an $\infty$-category that admits finite products and small colimits, such that the product preserves small colimits independently in each variable. We view $\mathcal{V}$ as a symmetric monoidal $\infty$-category $\mathcal{V}^{\otimes}$ via Cartesian product. 

Let $\mathcal{C}^{\otimes}$ be a presentably symmetric monoidal $\infty$-category and let $\phi: \mathcal{V}^{\otimes} \to \mathcal{C}^{\otimes}$ be a symmetric monoidal functor preserving small colimits.

Assume that the resulting tensor action of $\mathcal{V}$ on $\mathcal{C}$ is also enriched.

\begin{rmk}
    Observe that, under the above assumptions, for any fixed $M \in \mathcal{C}$ the functor:
    \[
        (-) \otimes (-) : \mathcal{V} \times \mathcal{C} \to \mathcal{C}
    \]
    preserves colimits independently in both variables. As $\mathcal{C}$ is presentable, for any $V \in \mathcal{V}$ the functor:
    \[
        V \otimes (-) : \mathcal{C} \to \mathcal{C}
    \]
    admits a right adjoint: we have all cotensors. On the other hand, without presentability assumptions on $\mathcal{V}$, for each $M \in \mathcal{C}$ the functor:
    \[
        (-) \otimes M : \mathcal{V} \to \mathcal{C}
    \]
    does not have a right adjoint in general, so we must additionally ask for morphism objects to exist.
\end{rmk}

Finally, let $\mathcal{O}^{\otimes}$ be an $\infty$-operad in $\mathcal{C}$.

\begin{thm} \label{thm:X.otimes.R.cong.R}
    Let $X \in \mathcal{V}$ be an object with a pointing map $\mathbf{1}_{\mathcal{V}} \to X$.
    
    Let $R \in Alg_{\mathcal{O}}({\mathcal{C}})$ be an $O$-algebra in ${\mathcal{C}}$ with an equivalence $R \coprod R \simeq R$.
    
    Recall by remark \ref{rmk:tensor.1.is.final} that in this context the final object $\mathbf{1}_{\mathcal{V}}$ is also a tensor unit for $\mathcal{V}$ and a unit for the action of $\mathcal{V}$ on $Alg_{\mathcal{O}}({\mathcal{C}})$;  consider the natural morphism $X \to \mathbf{1}_{\mathcal{V}}$. 
    Then the map in $Alg_{\mathcal{O}}({\mathcal{C}})$:
    \[
        X \otimes R \to \mathbf{1}_{\mathcal{V}} \otimes R \cong R
    \]
    is an equivalence. 
\end{thm}

\begin{proof}
Let $E \in Alg_{\mathcal{O}}({\mathcal{C}})$ be any $O$ algebra in ${\mathcal{C}}$.
The following holds:
    
    \begin{lemma}\label{lemma:we.have.a.mono}
        We have a $(-1)$-truncated morphism (also called a monomorphism):
    \[
        Map_{Alg_{\mathcal{O}}({\mathcal{C}})}(R, E) \hookrightarrow {Map_{Alg_{\mathcal{O}}({\mathcal{C}})}(\mathbf{0}_{\mathcal{C}}, E)} 
    \]
    where $\mathbf{0}_{\mathcal{C}}$ is an initial object of $\mathcal(C)$.
    \end{lemma}
    For the theory of truncated objects and morphisms, we remand the reader to \cite[Section 5.5.6]{Lurie2009}. The proof of this lemma will be given afterwards.
    
    By proposition \ref{prop:tensor.algebras} the induced action of $\mathcal{V}^{\otimes}$ on $Alg_{\mathcal{O}}({\mathcal{C}})$ is tensored and enriched and by corollary \ref{cor:cotensors.on.algo(C)} it admits all cotensors. We then obtain the following diagram:

\begin{center}
\begin{tikzcd}
    Map_{Alg_{\mathcal{O}}({\mathcal{C}})}(R,E) \arrow[r, hook] \arrow[d] 
    & Map_{Alg_{\mathcal{O}}({\mathcal{C}})}(\mathbf{0}_{\mathcal{C}}, E) \cong \star
    \\ Map_{Alg_{\mathcal{O}}({\mathcal{C}})}(X \otimes R, E) \arrow[d, "\sim" {rotate=90, anchor=north}] &
    \\
    Map_{Alg_{\mathcal{O}}({\mathcal{C}})}(R, E^X) \arrow[r, hook]
    & Map_{Alg_{\mathcal{O}}({\mathcal{C}})}(\mathbf{0}_{\mathcal{C}}, E^X) \cong \star
\end{tikzcd}\addtocounter{tikzcomd}{1}
\end{center}

    We have in particular remarked that the mapping spaces out of an initial object are contractible and non-empty \cite[Proposition 1.2.12.4 and Remark 1.2.12.6]{Lurie2009}, so homotopy equivalent to a point.
 
    This implies that $Map_{Alg_{\mathcal{O}}({\mathcal{C}})}(R, E)$ and $Map_{Alg_{\mathcal{O}}({\mathcal{C}})}(R, E^X)$ are contractible as well. In fact, from \cite[Definition 5.5.6.8]{Lurie2009}, a morphism of spaces $f:X \to Y$ is $(-1)$-truncated morphism if and only if the fibres over any point of $Y$ are $(-1)$-truncated, which means that they are contractible. In our case $Map_{Alg_{\mathcal{O}}({\mathcal{C}})}(\mathbf{0}_{\mathcal{C}}, E)$ (or $Map_{Alg_{\mathcal{O}}({\mathcal{C}})}(\mathbf{0}_{\mathcal{C}}, E^X)$) is equivalent to a single point, so we get our assertion.
    
    This allows only two choices for the mapping space $Map(R,E)$: it is either equivalent to a point (in which case we say that $E$ is $R$-local) or empty. This holds analogously for the space  $Map(R, E^X)$. We show:
    
    \begin{clm}\label{clm:E.R.local.iff.EX}
        $E$ is $R$-local $\iff E^X$ is $R$-local.
    \end{clm}
        
    In fact, observe that according to \ref{rmk:unit}:
    \begin{itemize}
        \item By using the collapse $X \to 1_{\mathcal{V}}$, given any map $R \to E$, we have:
        \[
            R \to E \cong E^{1_{\mathcal{V}}} \to E^X
        \]
        \item By using the pointing $1_{\mathcal{V}} \to X$, given any map $R \to E^X$, we have:
        \[
            R \to E^X \to E^{\star} \cong E.
        \]
    \end{itemize}
    Hence, we have a map from $R \to E$ if and only if we have a map $R \to E^X$, which proves \ref{clm:E.R.local.iff.EX}. We can conclude that, for all $E$, $Map(R,E) \cong Map(R \otimes X, E)$,  as we wished to show.
\end{proof}

\begin{proof}[Proof of Lemma \ref{lemma:we.have.a.mono}]
 We recall the following results.

    \begin{lemma}{\cite[Example 5.5.6.13]{Lurie2009}}
        A morphism $f:X \to Y$ in an $\infty$-category $\mathcal{D}$ is $(-2)$-truncated if and only if it is an equivalence.
    \end{lemma}

    \begin{lemma}{\cite[Lemma 5.5.6.15]{Lurie2009}} \label{lemma:lurie09.5.5.6.15}
    Let $\mathcal{D}$ be an $\infty$-category that admits finite limits and let $k \geq -1$ be an integer. A morphism $f:D \to D'$ in $\mathcal{D}$ is $k$-truncated if and only if the diagonal $\delta: D \to D \times_{D'} D $ is $(k-1)$-truncated.        
    \end{lemma}
        
    Hence, to prove that we have a monomorphism $Map_{Alg_{\mathcal{O}}({\mathcal{C}})}(R, E) \hookrightarrow Map_{Alg_{\mathcal{O}}({\mathcal{C}})}(\mathbf{0}_{\mathcal{C}}, E)$, we are reduced to showing that the diagonal:
    \[
        Map_{Alg_{\mathcal{O}}({\mathcal{C}})}(R, E) \to Map_{Alg_{\mathcal{O}}({\mathcal{C}})}(R, E) \times_{Map_{Alg_{\mathcal{O}}({\mathcal{C}})}(\mathbf{0}_{\mathcal{C}}, E)} Map_{Alg_{\mathcal{O}}({\mathcal{C}})}(R, E)
    \]
    is $(-2)$-truncated, in other words, an equivalence. By direct application of \cite[Corollary 4.3.25]{Land2021}:
    \begin{multline*}
        Map_{Alg_{\mathcal{O}}({\mathcal{C}})}(R, E) \times_{Map_{Alg_{\mathcal{O}}({\mathcal{C}})}(\mathbf{0}_{\mathcal{C}}, E)} Map_{Alg_{\mathcal{O}}({\mathcal{C}})}(R, E) \\
        \cong Map_{Alg_{\mathcal{O}}({\mathcal{C}})}(R \coprod_{\mathbf{0}_{\mathcal{C}}} R, E) 
        \cong Map_{Alg_{\mathcal{O}}({\mathcal{C}})}(R, E)
    \end{multline*}
\end{proof}

\clearpage

\section{The motivic context and \texorpdfstring{$MHH(M\mathbb{Q})$}{MHH(MQ)}}\label{sec:motivic}

Let now $S$ be a nice (for instance, quasi compact and quasi separated, or better) scheme, and let $Spc(S)_{\bullet}$ and $SH(S)$ be Morel and Voevodski's categories of pointed motivic spaces and motivic spectra, seen as $\infty$-categories \cite{Robalo2015}. Observe that this pair of categories satisfies all the hypotheses of theorem \ref{thm:X.otimes.R.cong.R}. In particular, the infinite suspension functor:
\[
    \Sigma^{\infty}: Spc(S)_{\bullet} \to SH(S)
\]
extends to a colimit-preserving symmetric monoidal functor, once we endow these categories with the usual tensor product given by the smash. Hence motivic spectra are tensored on pointed motivic spaces. Moreover, both categories are presentable (presentably symmetric monoidal in fact), so the action admits right adjoints both when either the first or the second variable is fixed (cotensors and morphism objects, respectively). 

We apply theorem \ref{thm:X.otimes.R.cong.R} when $\mathcal{O}$ is chosen in particular to be the commutative operad $\mathbb{E}_{\infty}$. In this case, it is worth noting that $CAlg(SH(S))$ inherits a tensor product that identifies with the coproduct \cite[Proposition 3.2.4.7]{Lurie2017}: in this case our ``idempotent'' objects coincide in fact with the idempotent objects with respect to this monoidal structure.  We then obtain the following result:

\begin{prop}\label{prop:idempotent}
Let $X$ be a pointed motivic space and let $R \in CAlg(SH(S))$ be a highly commutative ring spectrum; suppose that $R$ is idempotent. Consider the collapse on the point morphism $X \to \star$. Then the induced map $X \otimes R \to \star \otimes R \cong R$ is an equivalence. 
\end{prop}

As an example of application, we now focus on the motivic ring spectrum $M\mathbb{Q}=M\mathbb{Z}\otimes \mathbb{Q}$ representing rational motivic cohomology. 

It is a highly commutative and idempotent motivic ring spectrum. In fact, commutativity of $M\mathbb{Q}$ descends from the commutativity of $M\mathbb{Z}$ (see, for instance, \cite{Spitzweck2012}). On the other hand, idempotence was proven by Cisinski and Deglise in \cite{CisDeg2019}; they first constructed an idempotent motivic ring spectrum $H_B$ \cite[Proposition 14.1.6]{CisDeg2019}, which was later identified with the usual $M\mathbb{Q}$ \cite[Corollary 16.1.7]{CisDeg2019}. We can then apply proposition \ref{prop:idempotent} to rational motivic cohomology; in particular, we have the following consequence:

\begin{cor} \label{cor:MQ.otimes.anything}
One has:
\begin{gather*}
 MHH(M\mathbb{Q}) \cong S^1_s \otimes M \mathbb{Q} \cong M \mathbb{Q} \qquad
 \mathbb{G}_m \otimes M\mathbb{Q} \cong M\mathbb{Q}\\
 \mathbb{P}^1 \otimes M\mathbb{Q} \cong M\mathbb{Q}.
\end{gather*}
\end{cor}

\begin{proof}
Just observe that $S^1_s,\,\mathbb{G}_m$ and $\mathbb{P}^1\cong S^1_s \wedge \mathbb{G}_m \in Spc(S)_{\bullet}$ are pointed motivic spaces; the equivalence $MHH(M\mathbb{Q}) \cong S^\mathbf{1}_s \otimes M \mathbb{Q}$ was already discussed in greater generality in the introduction.
\end{proof}

\begin{rmk}
One could get the result about motivic Hochschild homology straightforwardly from the definition: let $R$ be an idempotent, highly commutative ring spectrum. Then $R \wedge R \cong R$ by \cite[Proposition 4.8.2.9]{Lurie2017}. This implies:
\[
    MHH(R) \cong R \wedge_{R \wedge R} R \cong R \wedge_R R \cong R.
\]
\end{rmk}

What should be noticed here is that the tensor between $M\mathbb{Q}$ and $S^\mathbf{1}_s$, or $\mathbb{G}_m$, or any motivic sphere (as they are given by the smash products of those two) agree and give $M\mathbb{Q}$ as the product.

\begin{rmk}
    The structure of $MHH(M\mathbb{Q})$ is very different from that of $MHH(M\mathbb{Z}/p)$, which was described in \cite{DHOO2022}; in fact, when the base scheme is an algebraically closed field of characteristic different from $p$, $\pi_{\star} MHH(M\mathbb{Z}/p)$ is an algebra over $\pi_{\star} M\mathbb{Z}/p=\mathbb{Z}/p[\tau]$ in infinitely many independent generators. Now, $M\mathbb{Q}=M\mathbb{Z}[1/2,1/3,1/5,\ldots]$; observe that also the homotopy ring structure of $MHH(M\mathbb{Z}/p)$ greatly simplifies after inverting a certain homotopy class, namely, $\tau$ (see again \cite{DHOO2022} for a presentation of $\pi_{\star} MHH(M\mathbb{Z}/p[\tau^{-1}])$). 
\end{rmk}

\clearpage

\printbibliography[heading=bibintoc]

@article{DHOO2022,
	doi = {10.48550/ARXIV.2204.00441},
	url = {https://arxiv.org/abs/2204.00441},
	author = {Dundas, Bjørn Ian and Hill, Michael A. and Ormsby, Kyle and Østvær, Paul Arne},
	title = {Hochschild homology of mod-$p$ motivic cohomology over algebraically closed fields},
	publisher = {arXiv},
	year = {2022},
	copyright = {Creative Commons Attribution 4.0 International}
}

@book{Lurie2009,
	author = "Lurie, Jacob",
	title = "{Higher Topos Theory}",
	publisher = "Princeton University Press",
	address = "Princeton, NJ",
	series = "Annals of mathematics studies",
	year = "2009"
}

@book{Lurie2017,
	author = "Lurie, Jacob",
	title = "{Higher Algebra}",
	publisher = "Online self-publishing",
	year = "2017",
	url = "https://people.math.harvard.edu/~lurie/papers/HA.pdf",
	urldate ={2023-02-17}
}

@misc{Spitzweck2012,
	doi = {10.48550/ARXIV.1207.4078},
	url = {https://arxiv.org/abs/1207.4078},
	author = {Spitzweck, Markus},
	title = {A commutative $P^1$-spectrum representing motivic cohomology over Dedekind domains},
	publisher = {arXiv},
	year = {2012},
	copyright = {arXiv.org perpetual, non-exclusive license}
}

@book{CisDeg2019,
	doi = {10.1007/978-3-030-33242-6},
	url = {https://doi.org/10.1007/978-3-030-33242-6},
	year = 2019,
	publisher = {Springer International Publishing},
	author = {Denis-Charles Cisinski and Fr{\'{e}}d{\'{e}}ric D{\'{e}}glise},
	title = {Triangulated Categories of Mixed Motives}
}

@article{NikSch2018,
	author = {Thomas Nikolaus and Peter Scholze},
	title = {{On topological cyclic homology}},
	volume = {221},
	journal = {Acta Mathematica},
	number = {2},
	publisher = {Institut Mittag-Leffler},
	pages = {203 -- 409},
	year = {2018},
	doi = {10.4310/ACTA.2018.v221.n2.a1},
	URL = {https://doi.org/10.4310/ACTA.2018.v221.n2.a1}
}

@article{Robalo2015,
	title = {K-theory and the bridge from motives to noncommutative motives},
	journal = {Advances in Mathematics},
	volume = {269},
	pages = {399-550},
	year = {2015},
	issn = {0001-8708},
	doi = {https://doi.org/10.1016/j.aim.2014.10.011},
	url = {https://www.sciencedirect.com/science/article/pii/S0001870814003570},
	author = {Marco Robalo}
}

@misc{Lurie2018,
	title = {Kerodon},
	author = {Jacob Lurie},
	howpublished = {\url{https://kerodon.net}},
	year = {2018}
}

@misc{Heine2023a,
	title={A monadicity theorem for higher algebraic structures},
	author={Hadrian Heine},
    note={Update of \href{https://arxiv.org/abs/1712.00555}{1712.00555}, to appear},
	primaryClass={math.CT}
}

@article{Heine2023b,
	title = {An equivalence between enriched $\infty$-categories and $\infty$-categories with weak action},
	journal = {Advances in Mathematics},
	volume = {417},
	pages = {108941},
	year = {2023},
	issn = {0001-8708},
	doi = {https://doi.org/10.1016/j.aim.2023.108941},
	url = {https://www.sciencedirect.com/science/article/pii/S0001870823000841},
	author = {Hadrian Heine}
}

@misc{Haugseng2019,
	title={Algebras for enriched $\infty$-operads},
	author={Rune Haugseng},
	year={2019},
	eprint={1909.10042},
	archivePrefix={arXiv},
	primaryClass={math.AT}
}

@book{Land2021,
	title={Introduction to Infinity-Categories},
	author={Markus Land},
	isbn={9783030615246},
	series={Compact Textbooks in Mathematics},
	url={https://books.google.de/books?id=1sMqEAAAQBAJ},
	year={2021},
	publisher={Springer International Publishing}
}

@book{KriMay1995,
	author = {K\v{r}{\'\i}\v{z}, Igor and May, J. P.},
	title = {Operads, algebras, modules and motives},
	series = {Ast\'erisque},
	publisher = {Soci\'et\'e math\'ematique de France},
	number = {233},
	year = {1995},
	mrnumber = {1361938},
	zbl = {0840.18001},
	language = {en},
	url = {http://www.numdam.org/item/AST_1995__233__1_0/}
}

@book{AdaRos1994,
place={Cambridge},
series={London Mathematical Society Lecture Note Series},
title={Locally Presentable and Accessible Categories},
publisher={Cambridge University Press},
author={Adamek, J. and Rosicky, J.},
year={1994},
collection={London Mathematical Society Lecture Note Series}
}

\end{document}